\documentclass{article}

\usepackage[english]{babel}

\usepackage[letterpaper,top=2cm,bottom=2cm,left=3cm,right=3cm,marginparwidth=1.75cm]{geometry}

\usepackage{amsmath}
\usepackage{amsthm}
\usepackage{amssymb}
\usepackage{graphicx}
\usepackage{multirow}
\usepackage{float}
\usepackage{tikz, tikz-3dplot, pgfplots}
\usepackage{array}
\usepackage{tabulary}
\usepackage{dynkin-diagrams}
\usepackage[colorlinks=true, allcolors=blue]{hyperref}
\usepackage{enumerate}
\usepackage{longtable}
\usepackage{pdflscape}

\usetikzlibrary{calc,3d, perspective,positioning,quotes}

\newcolumntype{K}{>{\centering\arraybackslash}p{1.75cm}}
\newcolumntype{J}{>{\centering\arraybackslash}p{4.5cm}}

\newcommand{\pt}[1]{
    \filldraw[black] (#1) circle (0.75pt);
}

\newcommand{\ptb}[1]{
    \filldraw[black] (#1) circle (1.25pt);
}

\definecolor{backgray}{rgb}{0.3,0.3,0.3}
\newcommand{\pd}[1]{ 
    \filldraw[backgray] (#1) circle (0.75pt);
}

\newcommand{\ld}[2]{
    \draw[dash pattern=on 2pt off 2pt, gray, thin] (#1) -- (#2);
}

\newcommand{\ls}[2]{
    \draw[gray, semithick] (#1) -- (#2);
}

\newtheorem{theorem}{Theorem}[section]
\newtheorem{lemma}[theorem]{Lemma}
\newtheorem{corollary}[theorem]{Corollary}

\theoremstyle{definition}
\newtheorem{definition}[theorem]{Definition}

\usepackage[sorting=ynt]{biblatex}
\title{The complete set of noble polyhedra}
\author{Connor Hill}

\begin{document}
\maketitle

\begin{abstract}
    We provide a complete enumeration of the finite polyhedra that are noble, that is, polyhedra that are vertex transitive and facet transitive, through a computer-assisted proof which relates the possible existence of noble polyhedra to the roots of certain univariate or bivariate cubic polynomials. This is done through a parametrization of the orbits of point groups in $\mathbb{R}^3$, which naturally leads to a notion of \textit{criticality} for certain orbits of a given point group, and a related equivalence relation $\equiv_c$ on the set of orbits. We show that within a given equivalence class of orbits under this relation, noble facetings are equivalent, and furthermore that there are a finite number of such equivalence classes. This gives a finite number of test cases, which may be tested through computer search. In total, we find that there are exactly 146 noble polyhedra in addition to the previously known infinite sets of stephanoids (crown polyhedra) and disphenoids.
\end{abstract}

\section{Introduction}

The regular polyhedra, and more specifically the five convex regular polyhedra, have been a subject of study for millennia. These five convex regular polyhedra, known as the Platonic solids, have long been associated with and recognized by their high amounts of symmetry; in particular, they are the only convex polyhedra whose symmetry groups act transitively on their vertices, edges, and faces simultaneously.

It was not until the 17th century that the idea of regular polyhedra which were not convex ("regular star polyhedra") were first fully considered when Kepler discovered the small and great stellated dodecahedra $\{5/2,5\}$ and $\{5/2,3\}$ and recognized them as regular. These two polyhedra as well as the final two regular star polyhedra $\{5,5/2\}$ and $\{3,5/2\}$ were rediscovered by Poinsot in 1809, and soon after this Cauchy proved the set of 4 regular star polyhedra complete \cite[p.~5-6]{ARP}. Note that this generalization also allows polyhedra to have self-intersections in addition to being nonconvex (in fact, all four of these star polyhedra also contain self-intersections).

This enumeration by Cauchy led in part to investigations by Hess beginning in 1875 of generalizations of the definition of regular polyhedra, one of which being what we now refer to as a noble polyhedron \cite{hessinitial}. Noble polyhedra are defined as being simultaneously vertex transitive and face transitive, so clearly all regular polyhedra are also noble. When restricted to convex polyhedra, the only nonregular noble polyhedra are the \textit{disphenoids} which consist of four congruent (but not necessarily equilateral) triangles, and these have solely axial symmetry. The problem becomes much less trivial if nonconvexity is permitted. Through several papers (see \cite{hessinitial}, \cite{hess1}, \cite{hess2}) Hess discovered 16 new nonregular noble polyhedra and an infinite set of polyhedra known as \textit{stephanoids} with prismatic symmetry, all of which are nonconvex and contain self-intersections. Later Brückner (see \cite{bruckner1}, \cite{bruckner2}, \cite{bruckner3}, \cite{bruckner4}) discovered a total of 10 other new noble polyhedra between 1900 and 1907. There would be no more discoveries of new noble polyhedra satisfying the criteria in Section 2 for over a century, although several near-examples were later found by Grünbaum, who also notably coined the term "noble" \cite{hollowfaces}.

In 2008 Robert Webb would discover a new noble polyhedron as a faceting of the uniform snub cube while experimenting with his software Stella \cite{stella4Dexploration}. This would later spark a series of discoveries in 2020 of new noble polyhedra by Mikloweit and a user of the Stella forums, done by utilizing the tools provided by Stella software. In total, thay found 33 new examples of noble polyhedra. During the creation of this paper, two more noble polyhedra (sD-10.1 and sD-12.1) were independently discovered by Ben Klein.

These more recent discoveries bring into question exactly how many noble polyhedra may exist, and although these figures had been discovered through various different means, no rigorous algorithm to determine the completeness of the set was known. In this paper we resolve this problem by relating the existence of noble polyhedra to roots of polynomials which encode the property of certain sets of points being coplanar. Using results involving these polynomials,  we prove the following main theorems:
\begin{theorem}
     Under the definitions supplied in Section 2, the only noble polyhedra with a prismatic symmetry group are the stephanoids and disphenoids. Other than these, there are a finite number of noble polyhedra up to the canonical mapping given in Definition \ref{def:canonicalmapping}.
\label{thm:maincomputerless}
\end{theorem}
Using this result, we design and run a computer algorithm to enumerate the precise number of these exceptional cases, finding the following:
\begin{theorem}
    Under the definitions supplied in Section 2 and discounting the stephanoids and disphenoids, there are exactly 146 noble polyhedra up to similarity.
\label{thm:main}
\end{theorem}
We treat the cases of polyhedra with prismatic symmetry groups (the symmetries of a regular polygonal prism or a subgroup thereof) and nonprismatic symmetry groups (all other cases) separately as we use different approaches to prove their completeness.

\section{Definitions}

In order to state our results we clearly must define precisely what counts as a noble polyhedron, which turns out to not be as obvious as one may expect. As noted by Grünbaum in \cite{yourpolymypoly}, Hess and Brückner were vague in many cases on what they counted as a polyhedron and sometimes they even contradicted their own requirements. We define polyhedra in terms of \textit{abstract polyhedra} whose vertices have been \textit{realized} into 3-dimensional Euclidean space $\mathbb{R}^3$. We define these terms below. 

\subsection{Abstract polyhedra and realizations}
We first provide a few preliminary definitions for partially ordered sets that we will use to define an abstract polyhedron.

We briefly summarize some common terminology regarding partial orders. In the following, let $(P, \leq)$ be a partially ordered set (\emph{poset}). A \textit{flag} of $P$ is a maximal chain of $P$. If all flags of $P$ have $n$ elements, then the \textit{rank} of $P$ is $n-2$. Given two elements $f,g$ of a partially ordered set $P$ with $f \le g$, the \textit{section} from $f$ to $g$ of $P$ is the partially ordered subset $\{h \in P \mid f \le h \le g \}$. A \textit{bounded} poset is a partial order with a minimum and maximum element. These two elements are the \textit{improper} elements, and all other elements are \textit{proper}. A bounded poset $P$ is \textit{connected} if for any two proper elements $f,g \in P$, there is a sequence of proper elements $f_0, f_1, ... , f_n$ such that $f_0 = f$, $f_n = g$, and any $f_i, f_{i+1}$ are comparable.

\begin{definition}
An \textit{abstract polyhedron} is a partially ordered set $P$ satisfying the following properties:
\begin{itemize}
\item $P$ is bounded and has rank 3.
\item Every section of $P$ that is of rank 1 contains exactly 4 elements. 
\item All other sections are connected.
\end{itemize}
\label{def:abstractpoly}
\end{definition}

We can assign a rank to any element of $P$ by taking the rank of the section from the minimum element to that element. Elements of rank 0 are known as \textit{vertices}, and the set of all vertices of $P$ is known as its \textit{vertex-set}. Elements of rank 1 are called \textit{edges}, and elements of rank 2 are \textit{faces}. More generally, the bounded poset determined by any section of rank 2 will be an \textit{abstract polygon}, with an $n$-sided abstract polygon (an $n$-gon) containing $n$ proper elements that cover the minimum element (i.e. $n$ vertices).

A \textit{vertex figure} of an abstract polyhedron is a section between a vertex and the maximum element. Being a section of rank 2, the vertex-figure will always be an abstract polygon. If every face of an abstract polyhedron $P$ is a $p$-gon and every vertex figure of $P$ is a $q$-gon, then we notate the \textit{Schl\"afli type} of $P$ as $\{p,q\}$. We now move from discussing the abstract structure of polyhedra to discussing the \textit{realizations} of abstract polyhedra, and we define a polyhedron below.

\begin{definition}
    A \textit{polyhedron} $\mathcal{P}=(P,\le,f)$ is an abstract polyhedron $P$ equipped with a map $f$ from its vertices to Euclidean 3-space $\mathbb{R}^3$. This map is also known as a \textit{realization}.
\label{def:poly}
\end{definition}

This is a very general definition, so for the context of this paper, we also put the following restrictions on our polyhedra:

\begin{definition}
A polyhedron $\mathcal{P}$ is \textit{nondegenerate} if it satisfies the following conditions:

\begin{itemize}
\item $\mathcal{P}$ is finite.
\item The realization is \textit{faithful} - that is, the realization is injective and no two edges or faces have the exact same set of vertices.
\item For any face in $\mathcal{P}$, all vertices within the face are mutually coplanar ($\mathcal{P}$ is a planar polyhedron). This plane that passes through every vertex within the face is known as its \textit{face plane}.
\item No two faces sharing an edge may be coplanar (faces sharing an edge cannot share a face plane).
\end{itemize}
\label{def:nondegenerate}
\end{definition}

We include this last requirement because when this does ever occur, it is always possible to combine the two coplanar faces into one larger face containing all the edges found in only one of the two coplanar faces. Unless otherwise stated, we assume a polyhedron to be nondegenerate for the remainder of this paper. 

It is important to note that this definition of polyhedra is not the only one that could lead to interesting results when finding noble examples. We touch on this in more detail in Section 5.3, but other possibilities could include allowing \textit{skew} (non-planar) faces in our polyhedra, in particular Gr\"unbaum's framework of \textit{skeletal polyhedra} \cite{oldandnew}, which has been shown to give interesting results when analyzing the more restrictive set of regular polyhedra (see e.g. \cite{schulteskeletal}), or enumerating the set of noble polyhedron compounds.

\subsection{Properties}

A \textit{symmetry} of a polyhedron $\mathcal{P}$ with realization $f$ is an automorphism $\phi$ of the underlying abstract polytope of $\mathcal{P}$ such that the induced permutation on the vertices in the realization can be extended to an isometry of $\mathbb{R}^3$. The \textit{symmetry group} $\Gamma(\mathcal{P})$ of $\mathcal{P}$ is then the group generated by the set of symmetries of $\mathcal{P}$ under composition - as all finite symmetry groups fix at least one point in $\mathbb{R}^3$, $\Gamma(\mathcal{P})$ is by definition a point group due to our polyhedra being finite objects.

We then call a polyhedron \emph{vertex transitive} if its symmetry group acts transitively on the vertex-set of the polyhedron; that is, for any two vertices of the polyhedron there is a symmetry that maps the first vertex to the second. Similarly, a polyhedron is \textit{face transitive} if its symmetry group acts transitively on its faces. All vertex-transitive polyhedra have a circumsphere due to the fact that all the isometries of their symmetry group (which is a point group) contain at least one shared fixed point by definition; all the vertices must be the same distance from this fixed point due to the fact they act transitively on the symmetry group, making this fixed point the circumcenter. Without loss of generality, we may assume that the circumcenter of a vertex-transitive polyhedron is located at the origin for simplicity, unless stated otherwise.

We now define a noble polyhedron. Note that all noble polyhedra have a Schl\"afli type and a circumcenter.
\begin{definition}
    A polyhedron is \textit{noble} if it is both vertex-transitive and face-transitive.
\label{def:noble}
\end{definition}

There are some consequences to these definitions. For one, polyhedra such as the V-faced polyhedra mentioned in \cite{hollowfaces} will not be included; because their faces revisit vertices, either their realizations will not be faithful or their underlying structure will not be an abstract polyhedron (depending on how they are interpreted). Similar situations also occur for any other polyhedra whose faces revisit vertices. The wreath polyhedra described in the same paper will not be counted due to the fourth restriction for polyhedra, as the pairs of coplanar triangular faces that they have would combine to form the stephanoid polyhedra.

A \textit{compound} is generally any collection of polyhedra. Compounds can similarly be noble, where the symmetry group acts transitively on any two vertices or faces in the compound (even if the vertices or faces are from different polyhedra). The \textit{affine hull} of a set of points $S$ is the unique smallest affine set containing $S$, denoted $\operatorname{aff} S$. If $S$ is nonempty, this will either be a point, line, plane, or all of $\mathbb{R}^3$. The affine hull $\operatorname{aff} F$ of a face $F$ of a polyhedron will always be its face plane.

\begin{definition}
    Two polyhedra are \textit{facetings} of each other if they share vertex-sets. Similarly, a polyhedron is a faceting of a set of points $V$ if $V$ is the vertex-set of the polyhedron.
\label{def:faceting}
\end{definition}
The process of determining the set of facetings of a polyhedron or set of points is also known as \textit{faceting}. From now on, we will only be considering noble facetings of polyhedra.

Given a polyhedron $\mathcal{P}$, the \textit{dual} of $\mathcal{P}$ is the polyhedron $\mathcal{Q}$ which is abstractly defined by the order $a \le_\mathcal{Q} b$ if and only if $b \le_\mathcal{P} a$, so each face of $\mathcal{P}$ corresponds to a vertex in $\mathcal{Q}$ and vice versa. The realization of this dual polytope is then defined as follows: For each face $F \in \mathcal{P}$, we project the circumcenter $o$ of $\mathcal{P}$ onto the face plane of $F$ and call this point $p$. Then the corresponding vertex of the dual $\mathcal{Q}$ is obtained by inversion about the unit sphere centered at $o$. The dual is an involution, and a polyhedron is \textit{self-dual} if its dual is similar to itself. Duals provide us with another way to relate noble polyhedra other than faceting, as the following theorem proves.

\begin{theorem}
    If the dual of a noble polyhedron is a polyhedron, then it is noble.
\end{theorem}
\begin{proof}
    We show that the dual of a face-transitive polyhedron is vertex-transitive, from which our result follows immediately.
    
    Let $\mathcal{P}$ be a face-transitive polyhedron, $f$ and $g$ be faces of $\mathcal{P}$, and $\mathcal{Q}$ be its dual. Then there is a symmetry $\phi$ that maps $f$ to $g$, which must also map the face plane of $f$ to the face plane of $g$. Because the projection of the circumcenter $o$ onto a face plane will be the unique point on the face plane of minimum distance between it and $o$, and that all symmetries of $\mathcal{P}$ have $o$ as a fixed point, it follows that $\phi$ will map the projection of $o$ onto the face plane of $f$ to the projection of $o$ onto the face plane of $g$. As the spherical inversion of these two points will be the vertices of $\mathcal{Q}$ corresponding to $f$ and $g$ by definition, it follows that $\phi$ (an isometry) will map one vertex of $\mathcal{Q}$ to the other, and in general $\phi$ will map vertices of $\mathcal{Q}$ to vertices of $\mathcal{Q}$. 
    
    $\phi$ is also an abstract symmetry of the poset $(P,\le_\mathcal{P})$. The corresponding symmetry of $(P,\le_\mathcal{Q})$ must map vertices to vertices in the same way that $\phi$ does as an isometry. Because $\phi$ is a symmetry of $\mathcal{Q}$ abstractly and maps vertices to vertices as an isometry, this implies that $\phi \in \Gamma(\mathcal{Q})$; it then follows that $\mathcal{Q}$ is vertex-transitive as for every two vertices of $\mathcal{Q}$ there is a symmetry of the polyhedron mapping one to the other.
\label{thm:nobledual}
\end{proof}
We now begin working towards an algorithm for enumerating the set of all noble polyhedra. Our main method will first give a way to find the noble facetings of a given finite set of points in $\mathbb{R}^3$. From there, we parametrize the possible vertex-sets for noble polyhedra as orbits of point groups. From there, we show how ideas from the initial faceting method may be generalized, which allows the ability to facet multiple (and in many cases, a continuum of) candidate orbits simultaneously when they are equivalent under a certain relation $\equiv_c$. This allows a reduction of the search space of noble polyhedra from infinitely many possible sets of points to only a finite amount based on these equivalence classes. Finally, we use a computer algorithm to check all possible equivalence classes under $\equiv_c$ to enumerate the noble polyhedra.

\section{Methodology}

\subsection{Faceting}

Given a point group $G$ of isometries and a (planar) polygon $F$ with a vertex-set $V_F \subset V$ where $V$ is some orbit of $G$, we may generate an isometric realization $g(F)$ of $F$ with vertices $g(V_F) = \{g(v) \mid v \in V_F\}$ for any $g \in G$. We can see that $g(V_F) \subset V$ since $G$ acts transitively on $V$ ($V$ is an orbit of $G$). If we consider the set of all faces of the form $g(F)$, it is possible that a polyhedral structure will be induced. If the set of faces is also the set of faces of a polyhedron $\mathcal{P}$, then we can also determine the set of edges and the vertex-set of $\mathcal{P}$ (the vertex-set is $V$); this gives us all the necessary information to determine the entire structure of $\mathcal{P}$, including the realization.

\begin{definition}
Given a point group $G$ and a polygon $F$ whose vertices are a subset of some orbit $V$ of $G$, we define \textit{$G(F)$} to be the set of faces of the form $g(F)$ for all $g \in G$. If $G(F)$ induces a corresponding polyhedron, we may refer to the polyhedral form synonymously with the set of faces when applicable. \label{def:groupgen}
\end{definition}
\begin{lemma}
    Given a point group $G$ and a polygon $F$, if $G(F)$ is a polyhedron then it is a noble polyhedron.
\label{thm:generatepoly}
\end{lemma}
\begin{proof}
    Keeping in mind that $G$ must be a subgroup of the symmetry group of $G(F)$, we find that $G(F)$ must be vertex-transitive; its vertex-set $V$ is an orbit of $G$ and so $G$ acts transitively on $V$. Similarly, $G(F)$ is face-transitive because any face of $G(F)$ can be described as $g(F)$ for some $g \in G$, and therefore $G$ acts transitively on the set of faces as well.
\end{proof}

\begin{lemma}
   A noble polyhedron $\mathcal{P}$ is of the form $G(F)$ for some point group $G$ and some polygon $F$. \vspace{-13pt}
\label{thm:polygenerateable}
\end{lemma}
\begin{proof}
    Let $F$ be any face of $\mathcal{P}$ and let $G = \Gamma(\mathcal{P})$ be the symmetry group of $\mathcal{P}$. The vertex-set $V$ of $\mathcal{P}$ is acted transitively on by $G$ ($\mathcal{P}$ is vertex-transitive), so $V$ is an orbit of $G$. It follows that $G(F)$ has the same set of faces as $\mathcal{P}$ because $\mathcal{P}$ is face-transitive; it follows that they are equivalent.
\end{proof} 

This provides us with a way to find the set of all noble facetings of a given set of points $V$; by determining all point groups $G$ which act transitively on $V$, and all planar polygons $F$ whose vertices are within $V$, the noble facetings of $V$ will precisely be the cases where $G(F)$ is a polyhedron. We now move on to determining the possible planar polygons within a vertex-set $V$, as from this we may determine the set of noble facetings of a given vertex-set.
\begin{definition}
    Given a vertex-set $V$, we define the set $P(V)$ of \textit{planes of vertices} within $V$ to be the set of all subsets $p = \operatorname{aff}(s) \cap V$, where $s$ is a subset of $V$ containing 3 vertices.
\label{def:planes}
\end{definition}

\begin{theorem}
    If $F$ is a face with vertex-set $V_F$ of a noble polyhedron $\mathcal{P}$ with vertex-set $V$, then there is a unique element $p \in P(V)$ such that $V_F \subseteq p$.
\label{thm:allfacesinplane}
\end{theorem}
\begin{proof}
    As $F$ is a planar polygon, the vertices of $F$ are mutually coplanar and $V_F$ contains at least 3 points. Consider the face plane $\operatorname{aff}(V_F)$ of $F$ that passes through all the vertices of $F$ - clearly $p = \operatorname{aff}(V_F) \cap V \in P(V)$ as $|V_F| \ge 3$. This element is unique as at most one plane may pass through three or more points which are not colinear (as $\mathcal{P}$ is vertex transitive, this is impossible).
\end{proof}

This then allows us to find all of the possible polygons within a vertex-set $V$; these will correspond precisely with the cycles of vertices that are found inside a plane of vertices within $V$, and so the problem of finding noble polyhedra reduces to one of finding planes of vertices within vertex-sets. Additionally, we may narrow down the possible polygons within a given plane of vertices $p$ that generate a polyhedron under a group $G$ by finding the \textit{adjacency graph} of $p$ under $G$:

\begin{definition}
    Given a plane $p$ of vertices within a vertex-set $V$ and a symmetry group $G$ which acts transitively on $V$, the \textit{adjacency graph} of $p$ under $G$ is the graph whose vertices are those of $p$ and whose edges are the pairs $\{v_1,v_2\}$ where $v_1, v_2 \in p$ and there is a $g \in G$ such that $p \cap g(p) = \{v_1, v_2\}$.
\label{def:adjgraph}
\end{definition}

It turns out that a face of a noble polyhedron always has a corresponding cycle within the adjacency graph, as the following lemma proves. The converse turns out to be false - in many cases, the generated object is a noble polyhedron compound rather than a noble polyhedron (there are also more exotic counterexamples - see Table \ref{table:fissarypolyhedra}).

\begin{lemma}
    Given a polygon $F$ and a point group $G$, if $G(F)$ is a noble polyhedron and $p$ is the face plane of $F$, then neighboring vertices of $F$ are connected by an edge in the adjacency graph of $p$ under $G$.
\label{thm:polyfromcycle}
\end{lemma}
\begin{proof}
    Consider any edge $\{v_1, v_2\}$ of $F$. Because $G(F)$ forms a face-transitive polyhedron and every edge has exactly two faces incident to it, there will be another face $g(F)$ for some $g \in G$ which also contains the edge $\{v_1, v_2\}$. By Theorem \ref{thm:allfacesinplane}, there exist two planes of vertices $p,q$ that contain the vertices of $F$ and $g(F)$ respectively; these are distinct as otherwise $F$ and $g(F)$ would be adjacent and coplanar, and this contradicts Definition \ref{def:nondegenerate}. Then $p \cap q = \{v_1, v_2\}$ due to the fact that otherwise $p \cap q$ would contain at least 3 points in common, so we would have $p = q$. Therefore, there is an edge between $v_1$ and $v_2$ in the adjacency graph of $p$.
\end{proof}

It follows from the contrapositive of this lemma that if the adjacency graph of a plane of vertices contains no cycles of edges, then the plane of vertices cannot contain a face of a noble polyhedron. Note that this method does not necessarily rely on any numeric calculations - if we may determine the set of planes $P(V)$ and we know what the symmetry groups that act transitively on $V$ are, we can determine the facetings abstractly without considering the realization. This fact is important for our future computer-assisted proof.

For a given set of points $V$ and all groups $G_0,G_1,...,G_n$ which act transitively on $V$, we may determine the set of noble facetings of $V$ as follows: For each $p \in P(V)$, find all cycles $F$ in the adjacency graph of $p$. Checking if any $G_i(F)$ is polyhedral for $0 \le i \le n$ then enumerates all possible test cases.

\subsection{Vertices}

These strategies which we will use to assist with the enumeration of noble facetings so far only works for an individual set of vertices that we want to find the noble facetings of. We will work towards generalizing these processes to check an infinitude of different orbits of vertices at once, allowing us to narrow our currently infinite search space of possible vertices for noble polyhedra down to a finite number of test cases.

Informally, the way in which we will do this is to first group together the possible vertex-sets for noble polyhedra into categories with similar sets of planes (in the sense of Definition \ref{def:planes}) and symmetries. In particular, if two sets of vertices act transitively on different symmetry groups then they will be put in separate categories.

As the polyhedra we are considering are finite, their symmetry groups must be point groups. To notate these point groups, it will be convenient to utilize \textit{orbifold notation} \cite{orbifoldnotation} which gives us some information about the boundary of its fundamental domain. In orbifold notation, we notate a reflection on the boundary of a fundamental domain by a star *. If there are multiple reflections on the boundary of the domain we indicate an angle of $\frac{\pi}{n}$ between two reflection planes by a number $n$ that is written after the star. A number $m$ not found after a star indicates an order-$m$ axis of rotational symmetry on the boundary of the fundamental domain. One of the prismatic symmetries is also notated with a cross $\times$ - this indicates an improper rotation symmetry in the spherical case. We call this string of characters and numbers the \textit{orbifold symbol} of the point group.

\begin{figure}[H]
    \centering
    \tdplotsetmaincoords{70}{-26}

\scalebox{1.5}{
\begin{tikzpicture}[tdplot_main_coords]

  \draw[green, line cap=round, very thick] (0,0,0) -- (1.5,1.5,1.5);
  \draw[green, line cap=round, very thick] (0,0,0) -- (-1.5,1.5,1.5);
  \draw[green, line cap=round, very thick] (0,0,0) -- (1.5,1.5,-1.5);
  \draw[green, line cap=round, very thick] (0,0,0) -- (-1.5,1.5,-1.5);
  \draw[green, line cap=round, very thick] (0,0,0) -- (1.5,-1.5,-1.5);

  \draw[blue, line cap=round, very thick] (0,0,0) -- (2.5,0,0);
  \draw[blue, line cap=round, very thick] (0,0,0) -- (0,2.5,0);
  \draw[blue, line cap=round, very thick] (0,0,-0) -- (0,0,-2.5);

  \begin{scope}[canvas is zy plane at x=0]
     \draw [red,very thick,smooth,domain=20:200] plot ({2.022*cos(\x)}, {2.022*sin(\x)});
   \end{scope}

   \begin{scope}[canvas is zx plane at y=0]
     \draw [red,very thick,smooth,domain=50:230] plot ({2.022*cos(\x)}, {2.022*sin(\x)});
   \end{scope}

   \begin{scope}[canvas is xy plane at z=0]
     \draw [red,very thick,smooth,domain=-20:160] plot ({2.022*cos(\x)}, {2.022*sin(\x)});
   \end{scope}

  \shade[ball color = lightgray!40, opacity = 0.5] (0,0) circle (2cm);

  \draw[blue, line cap=round, very thick] (2,0,0) -- (0,0,0);
  \draw[blue, line cap=round, very thick] (0,2,0) -- (0,0,0);
  \draw[blue, line cap=round, very thick] (0,0,-2) -- (0,0,0);
  \draw[green, line cap=round, very thick] (-1.154701,-1.154701,-1.154701) -- (1.154701,1.154701,1.154701);
  \draw[green, line cap=round, very thick] (1.154701,-1.154701,-1.154701) -- (-1.154701,1.154701,1.154701);
  \draw[green, line cap=round, very thick] (-1.154701,1.154701,-1.154701) -- (1.154701,-1.154701,1.154701);
  \draw[green, line cap=round, very thick] (1.154701,1.154701,-1.154701) -- (-1.154701,-1.154701,1.154701);
  \draw[blue, line cap=round, very thick] (-2,0,0) -- (0,0,0);
  \draw[blue, line cap=round, very thick] (0,-2,0) -- (0,0,0);
  \draw[blue, line cap=round, very thick] (0,0,2) -- (0,0,0);
  
  \shade[ball color = lightgray!40, opacity = 0.8] (0,0) circle (2cm);
  \fill[fill=black] (0,0) circle (1pt);

  \shade[ball color = darkgray, opacity = 0.75] (0,-2,0) to [bend right=8.5]  (0,-1.414214,1.414214) to [bend right=8.5]
   (-1.154701,-1.154701,1.154701) to [bend right=8.5]
   (-1.414214,-1.414214,0) to [bend right=8.5] (0,-2,0);

  \begin{scope}[canvas is zy plane at x=0]
     \draw [red,very thick,smooth,domain=200:380] plot ({2.022*cos(\x)}, {2.022*sin(\x)});
   \end{scope}

   \begin{scope}[canvas is zx plane at y=0]
     \draw [red,very thick,smooth,domain=230:410] plot ({2.022*cos(\x)}, {2.022*sin(\x)});
   \end{scope}

   \begin{scope}[canvas is xy plane at z=0]
     \draw [red,very thick,smooth,domain=160:340] plot ({2.022*cos(\x)}, {2.023*sin(\x)});
   \end{scope}

   \draw[blue, line cap=round,very thick] (0,0,2.044) -- (0,0,2.5);
   \draw[blue, line cap=round,very thick] (0,-2.044,0) -- (0,-2.5,0);
   \draw[blue, line cap=round,very thick] (-2.044,0,0) -- (-2.5,0,0);
   
   \draw[green, line cap=round,very thick] (-1.154701,-1.154701,1.154701) -- (-1.5,-1.5,1.5);
   \draw[green, line cap=round,very thick] (1.154701,-1.154701,1.154701) -- (1.5,-1.5,1.5);
   \draw[green, line cap=round,very thick] (-1.154701,-1.154701,-1.154701) -- (-1.5,-1.5,-1.5);
   
\end{tikzpicture}
}
    \caption{Depiction of the symmetry group 3*2. Red circles indicate planes of reflection, axes of order-2 rotational symmetry are indicated in blue, and axes of order-3 rotational symmetry are indicated in green. A possible fundamental domain demonstrating the symbol 3*2 is also given.}
    \label{fig:pyritohedral}
\end{figure}
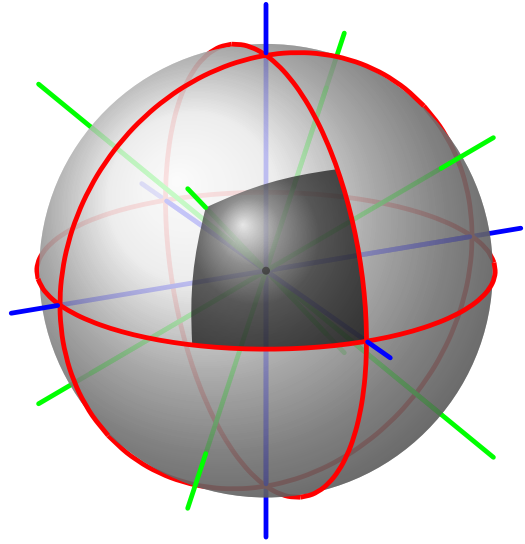

We split the point groups into two categories: 7 infinite sets of \textit{prismatic} symmetries, which have the same symmetries as the prism of a regular polygon or a subgroup thereof, and the remaining 7 \textit{nonprismatic} point groups. Namely, the 7 infinite sets of prismatic point groups are $*22n$, $2{*}n$, $22n$, $*nn$, $n*$, $n\times$, and $nn$ for any positive natural number $n$ and the 7 nonprismatic point groups are $332$, $*332$, $3{*}2$, $432$, $*432$, $532$, and $*532$. This completes the enumeration of the point groups (see Table III of \cite{orbifoldnotation}). Note that all of these groups are a normal subgroup of a finite three-dimensional \textit{reflection group} - a group that can be generated solely by reflections in Euclidean space.

\begin{lemma}
    All point groups in three dimensions are normal subgroups of a reflection group - namely $*332$, $*432$, $*532$, or $*22n$ for some $n$.
\label{thm:pointnormalsubgroups}
\end{lemma}

\begin{definition}
    A \textit{Coxeter matrix} is a symmetric $n\times n$ matrix $M$ whose entries are within $\{1,2,...\}\cup\{\infty\}$ and which satisfies the property that $m_{ii} = 1$ and otherwise $m_{ij} = m_{ji} \ge 2$ for all $i,j \le n$.
\label{def:coxetermatrix}
\end{definition}

\begin{definition}
    A \textit{Coxeter group} is a group defined by the presentation $\langle r_1,r_2,...,r_n \mid (r_ir_j)^{m_{ij}}=1\rangle$, where the $m_{ij}$ are the entries to a $n \times n$ Coxeter matrix $M$.
\label{def:coxetergroup}
\end{definition}

Notably, every reflection group is a representation of a Coxeter group, where each generator $r_i$ corresponds to a reflection in Euclidean space. For the three-dimensional reflection groups, these correspond to Coxeter groups with three generating reflections. We will use the following generating reflections, which can be verified to generate the correct reflection groups:

\begin{definition}
The three reflections for the $*332$ group will be given by the following matrix transformations:

\[
r_1 =
\begin{bmatrix}
  0 & 1 & 0 \\
  1 & 0 & 0 \\
  0 & 0 & 1
\end{bmatrix} \qquad\quad
r_2 =
\begin{bmatrix}
 0 & -1 & 0 \\
-1 & 0 & 0 \\
 0 & 0 & 1
\end{bmatrix} \qquad\quad
r_3 =
\begin{bmatrix}
1 & 0 & 0 \\
0 & 0 & 1 \\
0 & 1 & 0
\end{bmatrix}
\]
\label{def:332matrices}
\end{definition}

\begin{definition}
The three reflections for the $*432$ group will be given by the following matrix transformations:

\[
r_1 =
\begin{bmatrix}
-1 & 0 & 0 \\
 0 & 1 & 0 \\
 0 & 0 & 1
\end{bmatrix} \qquad\quad
r_2 =
\begin{bmatrix}
0 & 1 & 0 \\
1 & 0 & 0 \\
0 & 0 & 1
\end{bmatrix} \qquad\quad
r_3 =
\begin{bmatrix}
1 & 0 & 0 \\
0 & 0 & 1 \\
0 & 1 & 0
\end{bmatrix}
\label{def:432matrices}
\]
\end{definition}

\begin{definition}
The three reflections for the $*532$ group will be given by the following matrix transformations, where $\phi = \frac{\sqrt{5}+1}{2}$ is the golden ratio:
\[
r_1 =
\begin{bmatrix}
-1 & 0 & 0 \\
0 & 1 & 0 \\
0 & 0 & 1
\end{bmatrix} \qquad\quad
r_2 =
\begin{bmatrix}
1 & 0 & 0 \\
0 & -1 & 0 \\
0 & 0 & 1
\end{bmatrix} \qquad\quad
r_3 =
\frac{1}{2}
\begin{bmatrix}
1-\phi & -\phi & 1 \\
-\phi & 1 & \phi-1 \\
1 & \phi-1 & \phi
\end{bmatrix}
\]
\label{def:532matrices}
\end{definition}

\begin{definition}
The three reflections for the $*22n$ family of groups will be given by the following matrix transformations, where $\theta=\frac{2\pi}{n}$:
\[
r_1 =
\begin{bmatrix}
1 & 0 & 0 \\
0 & 1 & 0 \\
0 & 0 & -1
\end{bmatrix} \qquad\quad
r_2 =
\begin{bmatrix}
1 &  0 & 0 \\
0 & -1 & 0 \\
0 &  0 & 1
\end{bmatrix} \qquad\quad
r_3 =
\begin{bmatrix}
\cos(\theta) & \sin(\theta) & 0 \\
\sin(\theta) & -\cos(\theta) & 0 \\
0 & 0 & 1
\end{bmatrix}
\]
\label{def:22Nmatrices}
\end{definition}

We now categorize the orbits of the three-dimensional reflection groups up to similarity, as the vertices of every vertex-transitive (and by extension, noble) polyhedron are an orbit of their symmetry group. We start with the reflection groups, and then extend the notion to all point groups. Consider a reflection group $G$ which is a Coxeter group generated by the reflections $r_1,r_2,r_3$. Then for any three nonnegative real numbers $a,b,c$, we may place a unique point $p$ in the fundamental domain of $G$ that is distance $a$ from the $r_1$ mirror, distance $b$ away from the $r_2$ mirror, and distance $c$ away from the $r_3$ mirror. We call $p$ the \textit{generating point}, and from this we may determine the orbit of $p$; we notate this orbit $V_G(a,b,c)$ and call the values $a,b,c$ \textit{parameters}.

To extend this notion to the subgroups of the reflection groups, we know from Lemma \ref{thm:pointnormalsubgroups} that each point group $N$ is a normal subgroup of a reflection group $G$. Then $V_G(a,b,c)$ determines the orbits of $N$ up to similarity by finding the orbit of $N$ generated by the same generating point $p$ as was done for $V_G(a,b,c)$. The fact that $N$ is a normal subgroup guarantees the uniqueness of this subgroup orbit up to similarity, which we notate $V_N(a,b,c)$. We may also note that allowing at least one of these three parameters to be 0 will cause points in the orbit to combine together and/or for the resulting set of points to also act transitively on a symmetry group with a higher order of symmetry, which in both cases will have effects on faceting that we will need to consider separately.

As any generating point $p$ within the fundamental domain of a point group $G$ may be described in this way, this construction must encapsulate all possible orbits of $G$; however, there are still duplicates up to similarity as $V_G(a,b,c)$ will be similar to $V_G(ax,bx,cx)$ for all positive real $x$. To disregard these duplicates, we may assume without loss of generality that the last nonzero parameter of $V_G$ is 1. In combination with the fact that orbits with 0 as a parameter will be considered separately, this partitions each collection $V_N(a,b,c)$ of orbits up to similarity into 7 subsets that require separate consideration: $V_N(a,b,1), V_N(0,a,1), V_N(a,0,1), V_N(a,1,0), V_N(0,0,1),$ $V_N(0,1,0),$ and $V_N(1,0,0)$, where $a,b > 0$ are real numbers (we disregard $V_N(0,0,0)$ as it consists of only 1 point and does not need to be considered). We define the number of \textit{degrees of freedom} of a subset to be equal to the number of parameters in its definition (i.e. $a$ and possibly $b$).

Grouping identical subsets of orbits together (as some sets of points are transitive under multiple symmetry groups), for the nonprismatic symmetries we end up with 23 subsets that we will call \textit{orbit types}, and we define these below. We additionally give each orbit type a symbol as we will need to consisely refer to them later, especially when describing the results of the final enumeration.

\begin{table}[H]
  \caption{The 23 orbit types for nonprismatic symmetry groups.}
  \centering
\begin{tabular}{ | l | l | l | l | l | }
    \hline
    Symbol & Transitive under & Vertices & Possible definition as $V_N$ & Degrees of freedom \\ \hline
    T  & $332, *332$ & 4  & $V_{*332}(1,0,0)$ & 0 \\
    O  & $332, *332, 3{*}2, 432, *432$ & 6  & $V_{*432}(0,0,1)$ & 0 \\
    CO & $332, *332, 3{*}2, 432, *432$ & 12 & $V_{*432}(0,1,0)$ & 0 \\
    C  & $3{*}2, 432, *432$ & 8  & $V_{*432}(1,0,0)$ & 0 \\
    I  & $332, 3{*}2, 532, *532$ & 12 & $V_{*532}(0,1,0)$ & 0 \\
    ID & $532, *532$ & 30 & $V_{*532}(0,0,1)$ & 0 \\
    D  & $532, *532$ & 20 & $V_{*532}(1,0,0)$ & 0 \\
    tT$(a)$ & $332, *332$ & 12 & $V_{*332}(0,a,1)$ & 1 \\
    rT$(a)$ & $332, *332$ & 12 & $V_{*332}(a,1,0)$ & 1 \\
    rP$(a)$ & $332, 3{*}2$ & 12 & $V_{3{*}2}(0,a,1)$ & 1 \\
    tO$(a)$ & $*332, 432, *432$ & 24 & $V_{*432}(0,a,1)$ & 1 \\
    tC$(a)$ & $3{*}2, 432, *432$ & 24 & $V_{*432}(a,1,0)$ & 1 \\
    rC$(a)$ & $3{*}2, 432, *432$ & 24 & $V_{*432}(a,0,1)$ & 1 \\
    tI$(a)$ & $532, *532$ & 60 & $V_{*532}(0,a,1)$ & 1 \\
    tD$(a)$ & $532, *532$ & 60 & $V_{*532}(a,0,1)$ & 1 \\
    rD$(a)$ & $532, *532$ & 60 & $V_{*532}(a,1,0)$ & 1 \\
    sT$(a,b)$ & $332$ & 12 & $V_{332}(a,b,1)$ & 2 \\
    gT$(a,b)$ & $*332$ & 24 & $V_{*332}(a,b,1)$ & 2 \\
    gP$(a,b)$ & $3{*}2$ & 24 & $V_{3{*}2}(a,b,1)$ & 2 \\
    sC$(a,b)$ & $432$ & 24 & $V_{432}(a,b,1)$ & 2 \\
    gC$(a,b)$ & $*432$ & 48 & $V_{*432}(a,b,1)$ & 2 \\
    sD$(a,b)$ & $532$ & 60 & $V_{532}(a,b,1)$ & 2 \\
    gD$(a,b)$ & $*532$ & 120 & $V_{*532}(a,b,1)$ & 2 \\ \hline
  \end{tabular}
\label{table:nonprismatictypes}
\end{table}

\begin{figure}[H]
    \centering

\tdplotsetmaincoords{70}{-37}
\begin{tabular}{K K K K K}
    \begin{tikzpicture}[tdplot_main_coords]
  \ld{-0.57735,-0.57735,-0.57735}{0.57735,0.57735,-0.57735}
  \ld{0.57735,0.57735,-0.57735}{0.57735,-0.57735,0.57735}
  \ld{0.57735,0.57735,-0.57735}{-0.57735,0.57735,0.57735}
  
  \pd{0.57735,0.57735,-0.57735}
  
  \ls{-0.57735,-0.57735,-0.57735}{0.57735,-0.57735,0.57735}
  \ls{-0.57735,-0.57735,-0.57735}{-0.57735,0.57735,0.57735}
  \ls{-0.57735,0.57735,0.57735}{0.57735,-0.57735,0.57735}

  \pt{-0.57735,-0.57735,-0.57735}
  \pt{0.57735,-0.57735,0.57735}
  \pt{-0.57735,0.57735,0.57735}
\end{tikzpicture} T &

\begin{tikzpicture}[tdplot_main_coords]
  \ld{1,0,0}{0,1,0}
  \ld{1,0,0}{0,0,-1}
  \ld{0,1,0}{0,0,-1}
  
  \ls{-1,0,0}{0,-1,0}
  \ls{-1,0,0}{0,1,0}
  \ls{-1,0,0}{0,0,-1}
  \ls{-1,0,0}{0,0,1}
  \ls{1,0,0}{0,-1,0}
  \ls{1,0,0}{0,0,1}
  \ls{0,1,0}{0,0,1}
  \ls{0,-1,0}{0,0,-1}
  \ls{0,-1,0}{0,0,1}

  \pt{-1,0,0}
  \pt{1,0,0}
  \pt{0,-1,0}
  \pt{0,1,0}
  \pt{0,0,-1}
  \pt{0,0,1}
\end{tikzpicture} O &

\begin{tikzpicture}[tdplot_main_coords]
  \ld{0.707107,0.707107,0}{0.707107,0,0.707107}
  \ld{0.707107,0.707107,0}{0,0.707107,0.707107}
  \ld{0.707107,0.707107,0}{0.707107,0,-0.707107}
  \ld{0.707107,0,-0.707107}{0,0.707107,-0.707107}
  \ld{0.707107,0.707107,0}{0,0.707107,-0.707107}
  \ld{-0.707107,0,-0.707107}{0,0.707107,-0.707107}
  \ld{-0.707107,0.707107,0}{0,0.707107,-0.707107}
  \ld{0.707107,-0.707107,0}{0.707107,0,-0.707107}
  \ld{0.707107,0,-0.707107}{0,-0.707107,-0.707107}
  
  \pd{0.707107,0.707107,0}
  \pd{0.707107,0,-0.707107}
  \pd{0,0.707107,-0.707107}
  
  \ls{0.707107,0,0.707107}{0,0.707107,0.707107}
  \ls{-0.707107,0.707107,0}{-0.707107,0,0.707107}
  \ls{-0.707107,0,0.707107}{0,0.707107,0.707107}
  \ls{-0.707107,0.707107,0}{0,0.707107,0.707107}
  \ls{0.707107,-0.707107,0}{0.707107,0,0.707107}
  \ls{0.707107,0,0.707107}{0,-0.707107,0.707107}
  \ls{0.707107,-0.707107,0}{0,-0.707107,0.707107}
  \ls{-0.707107,-0.707107,0}{-0.707107,0,0.707107}
  \ls{-0.707107,0,0.707107}{0,-0.707107,0.707107}
  \ls{-0.707107,-0.707107,0}{0,-0.707107,0.707107}
  \ls{-0.707107,0.707107,0}{-0.707107,0,-0.707107}
  \ls{0.707107,-0.707107,0}{0,-0.707107,-0.707107}
  \ls{-0.707107,-0.707107,0}{-0.707107,0,-0.707107}
  \ls{-0.707107,0,-0.707107}{0,-0.707107,-0.707107}
  \ls{-0.707107,-0.707107,0}{0,-0.707107,-0.707107}

  \pt{0.707107,0,0.707107}
  \pt{0,0.707107,0.707107}
  \pt{-0.707107,0.707107,0}
  \pt{-0.707107,0,0.707107}
  \pt{0,-0.707107,0.707107}
  \pt{0.707107,-0.707107,0}
  \pt{-0.707107,-0.707107,0}
  \pt{-0.707107,0,-0.707107}
  \pt{0,-0.707107,-0.707107}
\end{tikzpicture} CO &

\begin{tikzpicture}[tdplot_main_coords]
  \ld{0.57735,-0.57735,-0.57735}{0.57735,0.57735,-0.57735}
  \ld{-0.57735,0.57735,-0.57735}{0.57735,0.57735,-0.57735}
  \ld{0.57735,0.57735,-0.57735}{0.57735,0.57735,0.57735}
  
  \pd{0.57735,0.57735,-0.57735}
  
  \ls{-0.57735,-0.57735,-0.57735}{0.57735,-0.57735,-0.57735}
  \ls{-0.57735,-0.57735,-0.57735}{-0.57735,0.57735,-0.57735}
  \ls{-0.57735,-0.57735,-0.57735}{-0.57735,-0.57735,0.57735}
  \ls{0.57735,-0.57735,-0.57735}{0.57735,-0.57735,0.57735}
  \ls{-0.57735,0.57735,-0.57735}{-0.57735,0.57735,0.57735}
  \ls{-0.57735,-0.57735,0.57735}{0.57735,-0.57735,0.57735}
  \ls{-0.57735,-0.57735,0.57735}{-0.57735,0.57735,0.57735}
  \ls{-0.57735,0.57735,0.57735}{0.57735,0.57735,0.57735}
  \ls{0.57735,-0.57735,0.57735}{0.57735,0.57735,0.57735}

  \pt{-0.57735,-0.57735,-0.57735}
  \pt{0.57735,-0.57735,-0.57735}
  \pt{-0.57735,0.57735,-0.57735}
  \pt{-0.57735,-0.57735,0.57735}
  \pt{0.57735,-0.57735,0.57735}
  \pt{-0.57735,0.57735,0.57735}
  \pt{0.57735,0.57735,0.57735}
\end{tikzpicture} C &

\begin{tikzpicture}[tdplot_main_coords]
  \ld{0.850651,0,0.525731}{0.850651,0,-0.525731}
  \ld{0.525731,0.850651,0}{-0.525731,0.850651,0}
  \ld{0,0.525731,-0.850651}{0,-0.525731,-0.850651}
  \ld{0,0.525731,0.850651}{0.525731,0.850651,0}
  \ld{0.525731,0.850651,0}{0.850651,0,0.525731}
  \ld{0,0.525731,-0.850651}{0.525731,0.850651,0}
  \ld{0.525731,0.850651,0}{0.850651,0,-0.525731}
  \ld{0.850651,0,-0.525731}{0,0.525731,-0.850651}
  \ld{0,0.525731,-0.850651}{-0.525731,0.850651,0}
  \ld{-0.850651,0,-0.525731}{0,0.525731,-0.850651}
  \ld{0.525731,-0.850651,0}{0.850651,0,-0.525731}
  \ld{0.850651,0,-0.525731}{0,-0.525731,-0.850651}
  
  \pd{0.525731,0.850651,0}
  \pd{0,0.525731,-0.850651}
  \pd{0.850651,0,-0.525731}
  
  \ls{0,0.525731,0.850651}{0,-0.525731,0.850651}
  \ls{-0.850651,0,0.525731}{-0.850651,0,-0.525731}
  \ls{0.525731,-0.850651,0}{-0.525731,-0.850651,0}
  \ls{0.850651,0,0.525731}{0,0.525731,0.850651}
  \ls{0,0.525731,0.850651}{-0.525731,0.850651,0}
  \ls{-0.525731,0.850651,0}{-0.850651,0,0.525731}
  \ls{-0.850651,0,0.525731}{0,0.525731,0.850651}
  \ls{0,-0.525731,0.850651}{0.525731,-0.850651,0}
  \ls{0.525731,-0.850651,0}{0.850651,0,0.525731}
  \ls{0.850651,0,0.525731}{0,-0.525731,0.850651}
  \ls{0,-0.525731,0.850651}{-0.525731,-0.850651,0}
  \ls{-0.525731,-0.850651,0}{-0.850651,0,0.525731}
  \ls{-0.850651,0,0.525731}{0,-0.525731,0.850651}
  \ls{-0.525731,0.850651,0}{-0.850651,0,-0.525731}
  \ls{0,-0.525731,-0.850651}{0.525731,-0.850651,0}
  \ls{0,-0.525731,-0.850651}{-0.525731,-0.850651,0}
  \ls{-0.525731,-0.850651,0}{-0.850651,0,-0.525731}
  \ls{-0.850651,0,-0.525731}{0,-0.525731,-0.850651}

  \pt{0,0.525731,0.850651}
  \pt{0.850651,0,0.525731}
  \pt{0,-0.525731,0.850651}
  \pt{-0.850651,0,0.525731}
  \pt{-0.525731,0.850651,0}
  \pt{0.525731,-0.850651,0}
  \pt{0,-0.525731,-0.850651}
  \pt{-0.850651,0,-0.525731}
  \pt{-0.525731,-0.850651,0}
\end{tikzpicture} I \vspace{3pt} \\
\end{tabular}
    \caption{Example orbits from each of the 23 nonprismatic orbit types. The edges of the convex hull of each orbit have also been drawn for visual clarity.}
    \label{fig:nonprismaticorbittypes}
\end{figure}
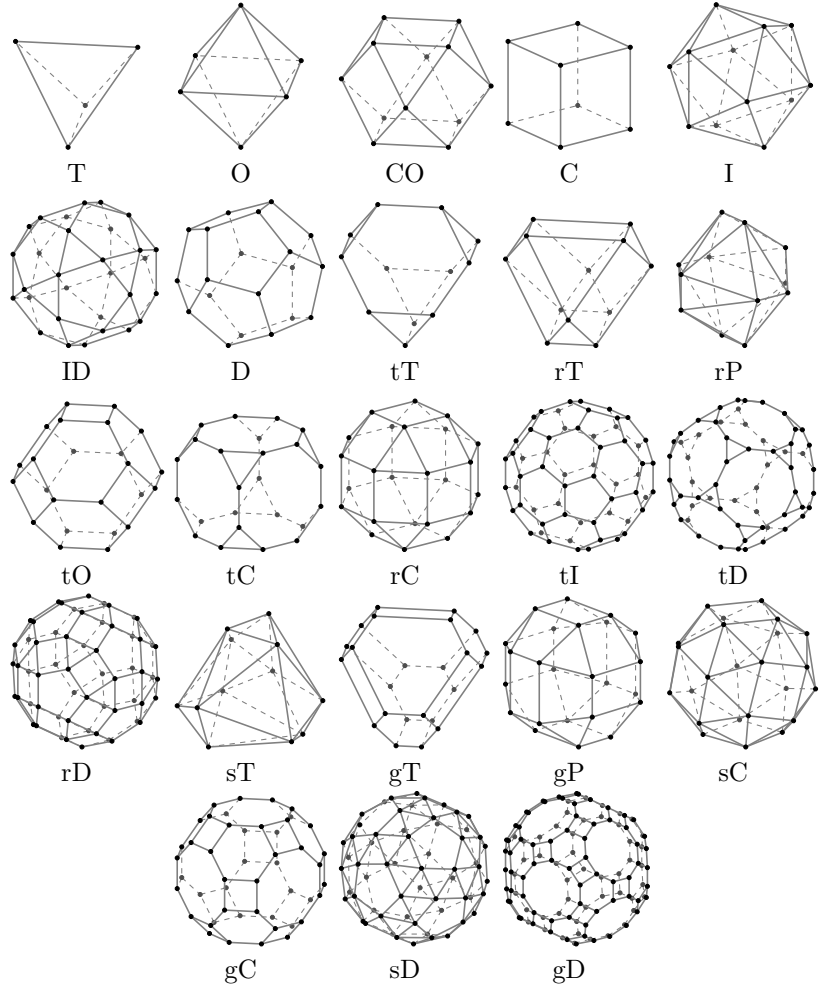

We may do the same process for the prismatic symmetries, although here there are some cases we do not include due to the fact that all points within these orbits are coplanar (these trivially cannot contain noble facetings, so we disregard them). Note that the families of groups $nn$ and $*nn$ contain no orbits that satisfy this condition.

\begin{table}[H]
    \caption{The 5 classes of orbit types for prismatic symmetry groups.}
  \centering
  \begin{tabular}{ | l | l | l | l | l | }
    \hline
    Symbol & Transitive under & Vertices & Possible definition as $V_N$ & Degrees of freedom \\ \hline
    P$n(a)$  $(n \ge 3)$ & $n*, 22n, *22n$       & $2n$ & $V_{*22n}(a,1,0)$ & 1 \\
    A$n(a)$  $(n \ge 2)$ & $n\times, 2{*}n, 22n$ & $2n$ & $V_{22n}(a,1,0)$ & 1 \\
    sA$n(a,b)$ $(n \ge 2)$ & $22n$                 & $2n$ & $V_{22n}(a,b,1)$ & 2 \\
    tP$n(a,b)$ $(n \ge 2)$ & $*22n$                & $4n$ & $V_{*22n}(a,b,1)$ & 2 \\
    tA$n(a,b)$ $(n \ge 2)$ & $2{*}n$               & $4n$ & $V_{2{*}n}(a,b,1)$ & 2 \\ \hline
  \end{tabular}
\label{table:prismatictypes}
\end{table}

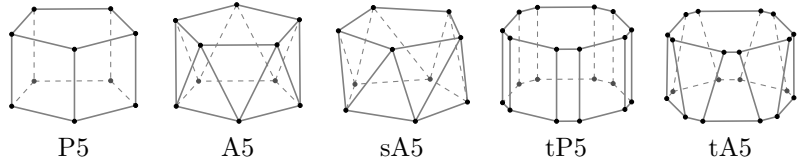
\begin{figure}[H]
    \centering

\tdplotsetmaincoords{70}{-37}
\begin{tabular}{K K K K K}
\begin{tikzpicture}[tdplot_main_coords]
    \ld{0,0.862104,0.506732}{0,0.862104,-0.506732}
    \ld{0.819909,0.266405,0.506732}{0.819909,0.266405,-0.506732}
    \ls{-0.819909,0.266405,0.506732}{-0.819909,0.266405,-0.506732}
    \ls{0.506732,-0.697457,0.506732}{0.506732,-0.697457,-0.506732}
    \ls{-0.506732,-0.697457,0.506732}{-0.506732,-0.697457,-0.506732}
    \ls{0,0.862104,0.506732}{0.819909,0.266405,0.506732}
    \ld{0,0.862104,-0.506732}{0.819909,0.266405,-0.506732}
    \ls{0,0.862104,0.506732}{-0.819909,0.266405,0.506732}
    \ld{0,0.862104,-0.506732}{-0.819909,0.266405,-0.506732}
    \ls{0.819909,0.266405,0.506732}{0.506732,-0.697457,0.506732}
    \ld{0.819909,0.266405,-0.506732}{0.506732,-0.697457,-0.506732}
    \ls{-0.819909,0.266405,0.506732}{-0.506732,-0.697457,0.506732}
    \ls{-0.819909,0.266405,-0.506732}{-0.506732,-0.697457,-0.506732}
    \ls{0.506732,-0.697457,0.506732}{-0.506732,-0.697457,0.506732}
    \ls{0.506732,-0.697457,-0.506732}{-0.506732,-0.697457,-0.506732}
    
    \pt{0,0.862104,0.506732}
    \pd{0,0.862104,-0.506732}
    \pt{0.819909,0.266405,0.506732}
    \pd{0.819909,0.266405,-0.506732}
    \pt{-0.819909,0.266405,0.506732}
    \pt{-0.819909,0.266405,-0.506732}
    \pt{0.506732,-0.697457,0.506732}
    \pt{0.506732,-0.697457,-0.506732}
    \pt{-0.506732,-0.697457,0.506732}
    \pt{-0.506732,-0.697457,-0.506732}
\end{tikzpicture} P5 &

\begin{tikzpicture}[tdplot_main_coords]
    \ld{0.8199093, 0.2664047, -0.506731}{0.8199093, -0.266404, 0.5067318}
    \ls{-0.0, -0.862103, 0.5067318}{-0.506731, -0.697456, -0.506731}
    \ld{-0.819909, 0.2664047, -0.506731}{-0.0, 0.8621037, -0.506731}
    \ls{-0.819909, -0.266404, 0.5067318}{-0.0, -0.862103, 0.5067318}
    \ls{-0.819909, 0.2664047, -0.506731}{-0.506731, -0.697456, -0.506731}
    \ls{0.8199093, -0.266404, 0.5067318}{0.5067318, -0.697456, -0.506731}
    \ls{0.8199093, -0.266404, 0.5067318}{-0.0, -0.862103, 0.5067318}
    \ld{0.8199093, 0.2664047, -0.506731}{-0.0, 0.8621037, -0.506731}
    \ls{0.5067318, -0.697456, -0.506731}{-0.506731, -0.697456, -0.506731}
    \ls{-0.819909, -0.266404, 0.5067318}{-0.819909, 0.2664047, -0.506731}
    \ls{-0.819909, -0.266404, 0.5067318}{-0.506731, 0.6974565, 0.5067318}
    \ls{0.5067318, 0.6974565, 0.5067318}{0.8199093, -0.266404, 0.5067318}
    \ls{0.5067318, 0.6974565, 0.5067318}{-0.506731, 0.6974565, 0.5067318}
    \ld{-0.0, 0.8621037, -0.506731}{-0.506731, 0.6974565, 0.5067318}
    \ld{0.8199093, 0.2664047, -0.506731}{0.5067318, -0.697456, -0.506731}
    \ls{0.5067318, -0.697456, -0.506731}{-0.0, -0.862103, 0.5067318}
    \ls{-0.819909, -0.266404, 0.5067318}{-0.506731, -0.697456, -0.506731}
    \ld{0.5067318, 0.6974565, 0.5067318}{0.8199093, 0.2664047, -0.506731}
    \ls{-0.819909, 0.2664047, -0.506731}{-0.506731, 0.6974565, 0.5067318}
    \ld{0.5067318, 0.6974565, 0.5067318}{-0.0, 0.8621037, -0.506731}

    \pt{-0.819909, -0.266404, 0.5067318}
    \pt{-0.819909, 0.2664047, -0.506731}
    \pt{0.5067318, 0.6974565, 0.5067318}
    \pd{0.8199093, 0.2664047, -0.506731}
    \pt{0.8199093, -0.266404, 0.5067318}
    \pt{0.5067318, -0.697456, -0.506731}
    \pd{-0.0, 0.8621037, -0.506731}
    \pt{-0.506731, 0.6974565, 0.5067318}
    \pt{-0.0, -0.862103, 0.5067318}
    \pt{-0.506731, -0.697456, -0.506731}
\end{tikzpicture} A5 &

\begin{tikzpicture}[tdplot_main_coords]
    \ls{0.8540377, 0.1176533, 0.5067318}{0.1520171, 0.8485950, 0.5067318}
    \ls{0.1520171, 0.8485950, 0.5067318}{-0.760085, 0.4068072, 0.5067318}
    \ls{-0.760085, 0.4068072, 0.5067318}{-0.621776, -0.597174, 0.5067318}
    \ls{-0.621776, -0.597174, 0.5067318}{0.3758071, -0.775881, 0.5067318}
    \ls{0.3758071, -0.775881, 0.5067318}{0.8540377, 0.1176533, 0.5067318}
    \ld{0.7600859, 0.4068072, -0.506731}{-0.152017, 0.8485950, -0.506731}
    \ld{-0.152017, 0.8485950, -0.506731}{-0.854037, 0.1176533, -0.506731}
    \ls{-0.854037, 0.1176533, -0.506731}{-0.375807, -0.775881, -0.506731}
    \ls{-0.375807, -0.775881, -0.506731}{0.6217761, -0.597174, -0.506731}
    \ld{0.6217761, -0.597174, -0.506731}{0.7600859, 0.4068072, -0.506731}

    \ld{0.8540377, 0.1176533, 0.5067318}{0.7600859, 0.4068072, -0.506731}
    \ld{0.7600859, 0.4068072, -0.506731}{0.1520171, 0.8485950, 0.5067318}
    \ld{0.1520171, 0.8485950, 0.5067318}{-0.152017, 0.8485950, -0.506731}
    \ld{-0.152017, 0.8485950, -0.506731}{-0.760085, 0.4068072, 0.5067318}
    \ls{-0.760085, 0.4068072, 0.5067318}{-0.854037, 0.1176533, -0.506731}
    \ls{-0.854037, 0.1176533, -0.506731}{-0.621776, -0.597174, 0.5067318}
    \ls{-0.621776, -0.597174, 0.5067318}{-0.375807, -0.775881, -0.506731}
    \ls{-0.375807, -0.775881, -0.506731}{0.3758071, -0.775881, 0.5067318}
    \ls{0.3758071, -0.775881, 0.5067318}{0.6217761, -0.597174, -0.506731}
    \ld{0.6217761, -0.597174, -0.506731}{0.8540377, 0.1176533, 0.5067318}
    
    \pt{0.8540377, 0.1176533, 0.5067318}
    \pd{0.7600859, 0.4068072, -0.506731}
    \pt{-0.375807, -0.775881, -0.506731}
    \pt{-0.621776, -0.597174, 0.5067318}
    \pd{-0.152017, 0.8485950, -0.506731}
    \pt{0.1520171, 0.8485950, 0.5067318}
    \pt{-0.760085, 0.4068072, 0.5067318}
    \pt{-0.854037, 0.1176533, -0.506731}
    \pt{0.6217761, -0.597174, -0.506731}
    \pt{0.3758071, -0.775881, 0.5067318}
\end{tikzpicture} sA5 &

\begin{tikzpicture}[tdplot_main_coords]
    \ls{-0.375807, -0.775881, -0.506731}{-0.621776, -0.597174, -0.506731}
    \ls{-0.760085, 0.4068072, -0.506731}{-0.760085, 0.4068072, 0.5067318}
    \ls{-0.375807, -0.775881, 0.5067318}{-0.621776, -0.597174, 0.5067318}
    \ld{0.8540377, 0.1176533, -0.506731}{0.7600859, 0.4068072, -0.506731}
    \ld{-0.152017, 0.8485950, -0.506731}{-0.152017, 0.8485950, 0.5067318}
    \ls{-0.375807, -0.775881, -0.506731}{0.3758071, -0.775881, -0.506731}
    \ls{-0.375807, -0.775881, 0.5067318}{0.3758071, -0.775881, 0.5067318}
    \ld{-0.152017, 0.8485950, -0.506731}{-0.760085, 0.4068072, -0.506731}
    \ls{-0.152017, 0.8485950, 0.5067318}{0.1520171, 0.8485950, 0.5067318}
    \ls{0.8540377, 0.1176533, 0.5067318}{0.7600859, 0.4068072, 0.5067318}
    \ls{-0.621776, -0.597174, -0.506731}{-0.854037, 0.1176533, -0.506731}
    \ls{0.3758071, -0.775881, -0.506731}{0.3758071, -0.775881, 0.5067318}
    \ls{-0.375807, -0.775881, -0.506731}{-0.375807, -0.775881, 0.5067318}
    \ls{-0.854037, 0.1176533, -0.506731}{-0.854037, 0.1176533, 0.5067318}
    \ld{0.8540377, 0.1176533, -0.506731}{0.8540377, 0.1176533, 0.5067318}
    \ld{0.1520171, 0.8485950, -0.506731}{0.1520171, 0.8485950, 0.5067318}
    \ls{-0.152017, 0.8485950, 0.5067318}{-0.760085, 0.4068072, 0.5067318}
    \ld{0.8540377, 0.1176533, -0.506731}{0.6217761, -0.597174, -0.506731}
    \ld{0.7600859, 0.4068072, -0.506731}{0.1520171, 0.8485950, -0.506731}
    \ls{0.8540377, 0.1176533, 0.5067318}{0.6217761, -0.597174, 0.5067318}
    \ls{-0.621776, -0.597174, -0.506731}{-0.621776, -0.597174, 0.5067318}
    \ls{0.6217761, -0.597174, -0.506731}{0.3758071, -0.775881, -0.506731}
    \ls{-0.621776, -0.597174, 0.5067318}{-0.854037, 0.1176533, 0.5067318}
    \ls{-0.760085, 0.4068072, -0.506731}{-0.854037, 0.1176533, -0.506731}
    \ls{0.7600859, 0.4068072, 0.5067318}{0.1520171, 0.8485950, 0.5067318}
    \ld{-0.152017, 0.8485950, -0.506731}{0.1520171, 0.8485950, -0.506731}
    \ls{0.6217761, -0.597174, 0.5067318}{0.3758071, -0.775881, 0.5067318}
    \ld{0.7600859, 0.4068072, -0.506731}{0.7600859, 0.4068072, 0.5067318}
    \ls{-0.760085, 0.4068072, 0.5067318}{-0.854037, 0.1176533, 0.5067318}
    \ls{0.6217761, -0.597174, -0.506731}{0.6217761, -0.597174, 0.5067318}
    
    \pd{0.8540377, 0.1176533, -0.506731}
    \pt{0.8540377, 0.1176533, 0.5067318}
    \pd{0.7600859, 0.4068072, -0.506731}
    \pt{0.7600859, 0.4068072, 0.5067318}
    \pt{-0.375807, -0.775881, -0.506731}
    \pt{-0.375807, -0.775881, 0.5067318}
    \pt{-0.621776, -0.597174, -0.506731}
    \pt{-0.621776, -0.597174, 0.5067318}
    \pd{-0.152017, 0.8485950, -0.506731}
    \pt{-0.152017, 0.8485950, 0.5067318}
    \pd{0.1520171, 0.8485950, -0.506731}
    \pt{0.1520171, 0.8485950, 0.5067318}
    \pt{-0.760085, 0.4068072, -0.506731}
    \pt{-0.760085, 0.4068072, 0.5067318}
    \pt{-0.854037, 0.1176533, -0.506731}
    \pt{-0.854037, 0.1176533, 0.5067318}
    \pt{0.6217761, -0.597174, -0.506731}
    \pt{0.6217761, -0.597174, 0.5067318}
    \pt{0.3758071, -0.775881, -0.506731}
    \pt{0.3758071, -0.775881, 0.5067318}
\end{tikzpicture} tP5 &

\begin{tikzpicture}[tdplot_main_coords]
    \ls{0.1406598, -0.838940, -0.5}{0.7544137, -0.393022, -0.5}
    \ls{-0.100721, 0.8446667, 0.5}{0.1007218, 0.8446667, 0.5}
    \ls{0.5779683, -0.624146, 0.5}{0.4149969, -0.742552, 0.5}
    \ls{0.5779683, -0.624146, 0.5}{0.8344505, 0.1652241, 0.5}
    \ls{-0.834450, 0.1652241, 0.5}{-0.772201, 0.3568085, 0.5}
    \ld{0.6069132, 0.5960395, -0.5}{0.8413463, -0.125471, -0.5}
    \ls{-0.834450, 0.1652241, 0.5}{-0.577968, -0.624146, 0.5}
    \ls{-0.606913, 0.5960395, -0.5}{-0.841346, -0.125471, -0.5}
    \ls{0.1007218, 0.8446667, 0.5}{0.7722010, 0.3568085, 0.5}
    \ls{0.1406598, -0.838940, -0.5}{-0.140659, -0.838940, -0.5}
    \ld{0.3793207, 0.7613951, -0.5}{-0.379320, 0.7613951, -0.5}
    \ld{0.3793207, 0.7613951, -0.5}{0.6069132, 0.5960395, -0.5}
    \ls{-0.841346, -0.125471, -0.5}{-0.754413, -0.393022, -0.5}
    \ls{0.4149969, -0.742552, 0.5}{-0.414996, -0.742552, 0.5}
    \ls{-0.772201, 0.3568085, 0.5}{-0.100721, 0.8446667, 0.5}
    \ld{-0.379320, 0.7613951, -0.5}{-0.606913, 0.5960395, -0.5}
    \ld{0.8413463, -0.125471, -0.5}{0.7544137, -0.393022, -0.5}
    \ls{0.8344505, 0.1652241, 0.5}{0.7722010, 0.3568085, 0.5}
    \ls{-0.140659, -0.838940, -0.5}{-0.754413, -0.393022, -0.5}
    \ls{-0.414996, -0.742552, 0.5}{-0.577968, -0.624146, 0.5}

    \ls{-0.834450, 0.1652241, 0.5}{-0.841346, -0.125471, -0.5}
    \ls{0.4149969, -0.742552, 0.5}{0.1406598, -0.838940, -0.5}
    \ls{-0.772201, 0.3568085, 0.5}{-0.606913, 0.5960395, -0.5}
    \ld{-0.100721, 0.8446667, 0.5}{-0.379320, 0.7613951, -0.5}
    \ld{0.1007218, 0.8446667, 0.5}{0.3793207, 0.7613951, -0.5}
    \ls{0.5779683, -0.624146, 0.5}{0.7544137, -0.393022, -0.5}
    \ls{-0.414996, -0.742552, 0.5}{-0.140659, -0.838940, -0.5}
    \ls{-0.577968, -0.624146, 0.5}{-0.754413, -0.393022, -0.5}
    \ld{0.8344505, 0.1652241, 0.5}{0.8413463, -0.125471, -0.5}
    \ld{0.7722010, 0.3568085, 0.5}{0.6069132, 0.5960395, -0.5}
    
    \pt{-0.834450, 0.1652241, 0.5}
    \pt{-0.772201, 0.3568085, 0.5}
    \pt{-0.100721, 0.8446667, 0.5}
    \pt{0.1007218, 0.8446667, 0.5}
    \pt{0.5779683, -0.624146, 0.5}
    \pt{0.4149969, -0.742552, 0.5}
    \pt{-0.414996, -0.742552, 0.5}
    \pt{-0.577968, -0.624146, 0.5}
    \pt{0.8344505, 0.1652241, 0.5}
    \pt{0.7722010, 0.3568085, 0.5}
    \pd{0.3793207, 0.7613951, -0.5}
    \pd{-0.379320, 0.7613951, -0.5}
    \pt{-0.606913, 0.5960395, -0.5}
    \pt{-0.841346, -0.125471, -0.5}
    \pd{0.6069132, 0.5960395, -0.5}
    \pd{0.8413463, -0.125471, -0.5}
    \pt{0.1406598, -0.838940, -0.5}
    \pt{0.7544137, -0.393022, -0.5}
    \pt{-0.140659, -0.838940, -0.5}
    \pt{-0.754413, -0.393022, -0.5}
\end{tikzpicture} tA5 \\

\end{tabular}
    
    \caption{Example orbits from each of the 5 classes of prismatic orbit types (with $n = 5$.) The edges of the convex hull of each orbit have also been drawn for visual clarity.}
    \label{fig:prismaticorbittypes}
\end{figure}

We may represent a member of an orbit type $T$ based on specifying values for its parameters (some of $a$ or $b$). We notate this as $T, T(a),$ or $T(a,b)$ depending on the number of degrees of freedom in $T$, and call the corresponding tuple of positive real numbers the \textit{location} of the orbit in $T$. One important consequence of the definition of an orbit type is that it induces a natural \textit{canonical mapping} between any two orbits within a given orbit type. We define this below.

\begin{definition}
    Let $V,U \in T$ be two members of the same orbit type, and let $p_V$ and $p_U$ be the generating points of $U$ and $V$ respectively. If $G$ is a point group that acts transitively on all members of $T$, then the \textit{canonical mapping} $f:V\mapsto U$ is the unique mapping such that for all $g \in G$, $f(g(p_V))=g(p_U)$.
\label{def:canonicalmapping}
\end{definition}

If $T$ has at least one degree of freedom and $v$ is a point in an orbit of $T$, we may then consider the function $v^T(a)$ or $v^T(a,b)$ which gives the point in $T(a)$ or $T(a,b)$ equivalent to $v$ under the canonical mapping. In fact, the coordinates of $v^T$ are affine in terms of their parameters - for a reflection group $G$ generated by mirrors $r_1,r_2,r_3$, let $u_i$ be the vector which is distance 1 away from $r_i$, within the fundamental domain of $G$, and on the intersection of the other two mirrors. Then $au_1+bu_2+cu_3$ is distance $a$ from $r_1$, distance $b$ from $r_2$, and distance $c$ from $r_3$. Substituting this into the fourth columns of Table \ref{table:nonprismatictypes} and Table \ref{table:prismatictypes} then gives the base point of the orbit, and with the above choices for the representations of the reflection groups, we may determine that the coordinates of each $v^T$ will be linear equations.

\subsection{Faceting the continuum}

We now move on to determining the possible noble facetings of members of a given orbit type, rather than a single orbit. For orbit types with zero degrees of freedom, this is trivial; as there is only a single member of the orbit type that could contain facetings and so we can simply use the method described in Section 3.1. Note that these orbits must be tested for facetings under all symmetry groups that act transitively on the orbit type, as it is possible for noble polyhedra to have a symmetry group which is a subgroup of the symmetry group of their vertices alone.

Usually, if $\mathcal{P}$ is a noble polyhedron with vertex-set $V$, and its orbit type $T$ has at least one degree of freedom, there are other orbits $U \in T$ where the canonical mapping $f:V\mapsto U$ does not generate a corresponding noble polyhedron with $U$ as its vertex-set. In these cases, the realization $f(\mathcal{P})$ violates some condition of Definition \ref{def:nondegenerate} that $\mathcal{P}$ did not; in particular, either the faces of $f(\mathcal{P})$ are no longer planar, or adjacent faces of $f(\mathcal{P})$ become coplanar with each other. Thus the only orbits $V \in T$ where these "exceptional" noble polyhedra like $\mathcal{P}$ can exist are when there exists a set of points which are coplanar in $V$ but not always coplanar in other orbits within the orbit type. We formalize this idea below, and we call such an orbit \textit{critical}.

For each orbit type $T$, we assign some enumeration $i_T:V \mapsto |V|$ for each $V \in T$ that is a bijection and is preserved by the canonical mapping (i.e. if $f:V\mapsto U$ is the canonical mapping and $v \in V$, then $i_T(v) = i_T(f(v))$). We will also use the following fact: given any four points $v_0,v_1,v_2,v_3 \in V$, the following determinant evaluates to zero if and only if the four points are coplanar:

\[
\begin{vmatrix}
    v_{0x} & v_{0y} & v_{0z} & 1 \\
    v_{1x} & v_{1y} & v_{1z} & 1 \\
    v_{2x} & v_{2y} & v_{2z} & 1 \\
    v_{3x} & v_{3y} & v_{3z} & 1
\end{vmatrix}
\]

This determinant is equal to six times the (signed) volume of the tetrahedron formed by the points $v_1,v_2,v_3,v_4$ (see e.g. \cite{tetvolume}), so the determinant evaluates to zero exactly when the points are coplanar.

\begin{definition}
    Given an orbit $V$ of points with $k$ vertices that is a member of some orbit type $T$, we define the \textit{volume configuration} $\operatorname{Conf}_V: k^4 \mapsto \mathbb{R}$ of $V$ to be the unique function such that for any $v_0,v_1,v_2,v_3 \in V$, the tuple $\operatorname{Conf}_V(i_T(v_0),i_T(v_1),i_T(v_2),i_T(v_3))$ equals the determinant given above. Similarly, we define $\operatorname{Conf}_T$ to be the function mapping a set of parameters to the volume configuration of the orbit at that location. We refer to elements in the range of these functions as \textit{entries} of a volume configuration.
\label{def:volumeconfiguration}
\end{definition}

For a given tuple $(i_0,i_1,i_2,i_3)$, $\operatorname{Conf}_V(i_0,i_1,i_2,i_3)$ turns out to be polynomial in terms of its parameters ($a$ and possibly $b$) as $V$ varies across some orbit type $T$. This is due to the fact that the coordinates of the points in the orbits of $T$ are linear equations in terms of $a$ and $b$ (see Section 3.2), and the determinant of a matrix with polynomial entries will also be a polynomial. Furthermore, since the last column of the matrix is all ones, the polynomial will be of degree at most 3. It is thus natural to interpret $\operatorname{Conf}_T$ as a function mapping tuples $(i_0,i_1,i_2,i_3)$ to polynomials whose variables are the parameters of $T$.

As such, if an orbit $V$ contains four coplanar points that are not coplanar for every member of its orbit type, then the corresponding entry of $\operatorname{Conf}_T$ will be nonzero and the location of $V$ will be a root of that entry. The same also applies in the opposite direction - every (positive) root of a nonzero entry corresponds to an orbit with additional coplanar vertices. This gives a way to formally describe the notion of criticality:

\begin{definition}
    An orbit $V$ within an orbit type $T$ is \textit{critical} if its location is a root of a nonzero entry of $\operatorname{Conf}_T$. Otherwise, we say it is \textit{typical}.
\label{def:critical}
\end{definition}

In fact, this naturally gives a transitive ordering on $T$ which in some sense compares the criticality of two orbits (which we describe below).

\begin{definition}
    Given two orbits $V,U$ in the same orbit type, we say that \textit{$V$ is critical relative to $U$} (written $V \ge_c U$) if an entry of $\operatorname{Conf}_U$ evaluating to 0 implies that the corresponding entry of $\operatorname{Conf}_V$ under the canonical mapping is also equal to 0. Similarly, we say \textit{$V$ is critical-equivalent to $U$} (written $V \equiv_c U$) if $V \ge_c U$ and $U \ge_c V$.
\label{def:criticalrelation}
\end{definition}

\begin{theorem}
    The typical orbits of an orbit type $T$ are exactly the $V$ such that $V \le_c U$ for all $U \in T$.
\label{thm:typicalminimum}
\end{theorem}
\begin{proof}
    Let $V=T(a_V,b_V)$ be a typical orbit, and $U=T(a_U,b_U)$ be some orbit of $T$. By definition, $V$ is not a root of any nonzero entry $c$ of $\operatorname{Conf}_T$, so $c(a_V,b_V) = 0  \iff c = 0 \implies c(a_U,b_U)=0$ and so $V \le_c U$. If $V \le_c U$ for all $U \in T$, then we may let $U$ be a typical orbit and $V$ cannot be the root of any nonzero entry of $\operatorname{Conf}_T$.
\end{proof}

Stated another way, this result says that the typical orbits of an orbit type $T$ are exactly those in the minimum equivalence class of $T$ under $\ge_c$.

\begin{theorem}
    An orbit type $T$ with one degree of freedom may only have a finite number of critical orbits.
\label{thm:1dfinitecritical}
\end{theorem}
\begin{proof}
    By definition, $T$ will be parameterized by a single variable $a$, and therefore the entries of $\operatorname{Conf}_T$ will be univariate polynomials. As every polynomial has a finite number of positive real roots and $\operatorname{Conf}_T$ has a finite number of entries, in total there are a finite number of roots across all of the volume configuration's nonzero polynomial entries. By definition, each root corresponds exactly with a critical orbit, so there are a finite number of critical orbits in the orbit type.
\end{proof}

\begin{theorem}
    Given two orbits $V$ and $U$ of the same orbit type and their canonical mapping $f:V \mapsto U$, if $V \equiv_c U$ then any faceting $\mathcal{P}$ of $V$ which is noble under a symmetry of the orbit type has a corresponding noble faceting $f(\mathcal{P})$ of $U$ provided by the canonical mapping between points.\label{thm:equivalencepreservesfaces}
\end{theorem}
\begin{proof}
    If $\mathcal{P}$ has faces with more than three sides, let $p$ be a face of $\mathcal{P}$ and $v_0,v_1,v_2,v_3 \in V$ be distinct vertices of $p$. Then the points are coplanar and the entry of the $v_i$ in $\operatorname{Conf}_V$ is equal to 0. Because $U \ge_c V$, by definition the entry of $\operatorname{Conf}_U$ corresponding to the points $f(v_0),f(v_1),f(v_2),f(v_3)$ must also equal 0. As this is true for any four points in $p$, the vertices of $f(p)$ are also all mutually coplanar. Therefore, all faces of $f(\mathcal{P})$ contain mutually coplanar vertices. The same applies automatically if the faces of $\mathcal{P}$ are triangles, as a triangle will always have a face plane.

    To show that $f(\mathcal{P})$ is a nondegenerate polyhedron, it only remains to show that no two coplanar faces of $f(p)$ share an edge. Assume that two faces $p,q$ of $f(\mathcal{P})$ are coplanar and share an edge, and let $S$ be the collective set of vertices of both $p$ and $q$. Then if any quadruplet of vertices of $S$ are coplanar the corresponding entry of $\operatorname{Conf}_U$ equals 0. Then, because $V \ge_c U$, by definition the entry of $\operatorname{Conf}_V$ corresponding to any quadruplet of vertices of $f^{-1}(S)$ will equal zero, and therefore all points of $f^{-1}(S)$ are coplanar. However, this means that $f^{-1}(p)$ and $f^{-1}(q)$ are coplanar faces sharing an edge, implying that $\mathcal{P}$ is cannot be a polyhedron as given in Definition \ref{def:nondegenerate}.

    Finally, note that by assumption $\Gamma(\mathcal{P})$ acts transitively on all members of the orbit type of $U$ and $V$, so all symmetries of $\mathcal{P}$ are also symmetries of $f(\mathcal{P})$. This implies that $f(\mathcal{P})$ is also a noble polyhedron.
\end{proof}

\begin{corollary}
    Given two orbits $V$ and $U$, if $V \equiv_c U$ then they admit the same number of noble facetings up to similarity.
\label{thm:equivalencepreservesfacetings}
\end{corollary}

This last theorem and corollary immediately allow us to determine the number of noble facetings of the orbits of any equivalence class under $\equiv_c$ in a finite amount of time: if we choose any member $V$ of the equivalence class and facet that orbit for noble polyhedra, all facetings $\mathcal{P}$ of $V$ immediately also enumerate all facetings of all other orbits in the equivalence class through the canonical mapping (although it has yet to be determined if one can effectively find $V$). Therefore, we will prove that each orbit type $T$ contains a finite number of equivalence classes under $\equiv_c$ as this reduces our search space from an infinite number of cases to a finite amount.

Theorem \ref{thm:1dfinitecritical} already proves that there are a finite number of equivalence classes for orbit types with one degree of freedom, so we only need to prove this is the case for the orbit types $T$ with two degrees of freedom. As before, the critical orbits of the orbit type occur when their parameters provide a root to a nonzero entry of $\operatorname{Conf}_T$, however unlike the case with one degree of freedom each entry corresponds to a polynomial in two variables rather than one. If we create a plot of these zero sets with the parameters $a,b$ as our axes, these zero sets will correspond to cubic plane curves rather than a finite collection of roots. We provide an example below.

\begin{figure}[H]
    \centering

\begin{tikzpicture}
    \begin{axis}[
            axis x line=bottom,
            axis y line=left,
            axis line style={->,thick},
            xlabel={$a$},
            ylabel={$b$},
            xtick distance=1,
            ytick distance=1,
            ylabel style={yshift=-15pt},
            xmin=0,xmax=5.2,
            ymin=0,ymax=5.2,
            grid=both,
            axis equal image
            ]
            \addplot[mark=none, domain=0:5.2, samples = 60]{ x };
            \addplot[mark=none, domain=0:5.2, samples = 60]{ 1/x };

            \addplot[mark=none, domain=0:5.2, samples = 60]{ x^2 - 1 };
            \addplot[mark=none, domain=0:5.2, samples = 60]({ x^2 - 1 },{x});

            \addplot[mark=none, domain=0:5.2, samples = 60]{ x^2 + x };
            \addplot[mark=none, domain=0:5.2, samples = 60]({ x^2 + x },{x});

            \addplot[mark=none, domain=0:5.2, samples = 60]{ (x+1)/x };
            \addplot[mark=none, domain=0:5.2, samples = 60]({ (x+1)/x },{x});

            \addplot[mark=none, domain=0:5.2, samples = 60]{1/2*(sqrt(x^2 + 6*x + 1) - x - 1)};
            \addplot[mark=none, domain=0:5.2, samples = 60]({1/2*(sqrt(x^2 + 6*x + 1) - x - 1)},{x});

            \addplot[mark=none, domain=0:5.2, samples = 60]{ (sqrt(1/x^2 + 4*x + 4/x + 4)*x + 1)/(2*x) };

            \addplot[mark=none, domain=0:5.2, samples = 60]{(7*x^3 + 15*x^2 + 3*sqrt(3)*sqrt(3*x^6 + 6*x^5 + 23*x^4 + 36*x^3 + 36*x^2 + 4*x) + 30*x + 2)^(1/3)/(3*2^(1/3)) - (2^(1/3)*(2*x^2 - x - 1))/(3*(7*x^3 + 15*x^2 + 3*sqrt(3)*sqrt(3*x^6 + 6*x^5 + 23*x^4 + 36*x^3 + 36*x^2 + 4*x) + 30*x + 2)^(1/3)) + 1/3*(-x - 2)};
            \addplot[mark=none, domain=0:5.2, samples = 60]({(7*x^3 + 15*x^2 + 3*sqrt(3)*sqrt(3*x^6 + 6*x^5 + 23*x^4 + 36*x^3 + 36*x^2 + 4*x) + 30*x + 2)^(1/3)/(3*2^(1/3)) - (2^(1/3)*(2*x^2 - x - 1))/(3*(7*x^3 + 15*x^2 + 3*sqrt(3)*sqrt(3*x^6 + 6*x^5 + 23*x^4 + 36*x^3 + 36*x^2 + 4*x) + 30*x + 2)^(1/3)) + 1/3*(-x - 2)},{x});

            \addplot[mark=none, domain=0:5.2, samples = 60]{(-x^2 + sqrt((x^4 + 2*x^3 + 7*x^2 + 6*x + 1)/(x + 1)^2)*x + sqrt((x^4 + 2*x^3 + 7*x^2 + 6*x + 1)/(x + 1)^2) - x - 1)/(2*(x + 1))};
            \addplot[mark=none, domain=0:5.2, samples = 60]({(-x^2 + sqrt((x^4 + 2*x^3 + 7*x^2 + 6*x + 1)/(x + 1)^2)*x + sqrt((x^4 + 2*x^3 + 7*x^2 + 6*x + 1)/(x + 1)^2) - x - 1)/(2*(x + 1))},{x});

            \addplot[mark=none, domain=0:5.2, samples = 60]{ 1/2*(x^2 + sqrt(x^4 + 6*x^3 + 7*x^2 + 2*x + 1) + x - 1) };
            \addplot[mark=none, domain=0:5.2, samples = 60]({ 1/2*(x^2 + sqrt(x^4 + 6*x^3 + 7*x^2 + 2*x + 1) + x - 1) },{x});

            \addplot[mark=none, domain=0:5.2, samples = 60]{ 1/2*(x^2 + sqrt(x^4 + 6*x^3 + 7*x^2 + 2*x + 1) + x - 1) };
            \addplot[mark=none, domain=0:5.2, samples = 60]({ 1/2*(x^2 + sqrt(x^4 + 6*x^3 + 7*x^2 + 2*x + 1) + x - 1) },{x});

            \addplot[mark=none, domain=0:5.2, samples = 60]{ 1/2*(x^2 + sqrt(x^4 + 6*x^3 + 7*x^2 + 6*x + 1) + x - 1) };
            \addplot[mark=none, domain=0:5.2, samples = 60]({ 1/2*(x^2 + sqrt(x^4 + 6*x^3 + 7*x^2 + 6*x + 1) + x - 1) },{x});
            
            \addplot[very thick, red, mark=none, domain=0:5.2, samples = 60]({ (sqrt(1/x^2 + 4*x + 4/x + 4)*x + 1)/(2*x) },{x});
    \end{axis}
\end{tikzpicture}

    \caption{Plot of the critical orbits of the orbit type gT (view restricted to $a,b < 5$). A critical curve is drawn in red and all other critical orbits are drawn in black.}
    \label{fig:gTcritgraph}
\end{figure}
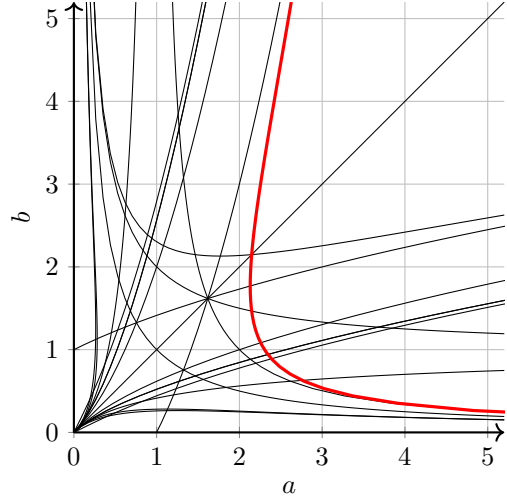

We call these cubic plane curves the \textit{critical curves} of the orbit type $T$. If two orbits $V$ and $U$ have locations which lie on exactly the same critical curves, then it must be that $V \equiv_c U$; by definition, each entry of a volume configuration corresponds to at most one critical curve and so these two properties are equivalent. However, note that in general each critical curve will correspond with many entries of the volume configuration.

Some unequal critical curves may overlap at an infinite number of points if they share a common factor. In these cases, the region of intersection will either be a line or a conic section. Otherwise, critical curves must intersect at only a finite number of points.

\begin{theorem}
    Any orbit type $T$ has a finite number of equivalence classes under $\equiv_c$.
\label{thm:finiteequiv}
\end{theorem}
\begin{proof}
    Theorem \ref{thm:1dfinitecritical} already proves this for the case of 1 degree of freedom, and the case for 0 degrees of freedom is trivial. Therefore, we may assume $T$ has two degrees of freedom.

    Let $f$ be the polynomial of a critical curve of $T$ and let $s$ be an irreducible factor of $f$ over the real numbers. Then each root $p$ of $s$ either lies on a critical curve which does not have $s$ as a factor (i.e. there is some other critical curve which intersects $s$ at $p$), or all critical curves that $p$ lies on have $s$ as a factor. There may only be a finite number of roots of $s$ which satisfy the former criterion; each critical curve which does not have $s$ as a factor may only intersect the zero set of $s$ at a finite number of points (as critical curves are cubic plane curves), and because there are a finite number of critical curves in total there may only be a finite number of intersections in total.

    All roots $p$ of $s$ which satisfy the latter case are equivalent under $\equiv_c$ due to the fact that the critical curves they lie on are exactly those whose polynomials have $s$ as a factor, so therefore they lie on exactly the same critical curves; it follows that they are critical-equivalent. Therefore, combining both cases, the zero set of $s$ must have a finite number of equivalence classes under $\equiv_c$. Because each polynomial $f$ of a critical curve of $T$ can only have up to three irreducible factors over the reals and there are a finite number of critical curves of $T$, the set of critical orbits of $T$ has a finite number of equivalence classes under $\equiv_c$. All typical orbits are in the same equivalence class by definition, so therefore there are a finite number of equivalence classes for the entirety of $T$.
\end{proof}

As a corollary, this proves Theorem \ref{thm:maincomputerless}.

Ideally, we would like to check for noble facetings of these equivalence classes without having to find an example member of each equivalence class and needing to check for facetings numerically. This is equivalent to the problem of determining which combinations of entries of $\operatorname{Conf}_T$ can evaluate to zero at once, as from this we may determine exactly which points are coplanar with each other, and this gives all the information needed to determine the facetings of the corresponding orbits, as implied by Corollary \ref{thm:equivalencepreservesfacetings} and the results of Section 3.1.

In an orbit type $T$ with one degree of freedom, the ability for two entries $f,g \in \operatorname{Conf}_T$ to evaluate to zero simultaneously is by definition equivalent to the fact that they share a root over the positive real numbers, and this implies that the polynomials share a common factor. Intuitively, these two polynomials 'split' into at least three equivalence classes given by the roots of $\gcd(f,g), \frac{f}{\gcd(f,g)},$ and $\frac{g}{\gcd(f,g)}$, which will all themselves be polynomials.

\begin{lemma}
    Let $s$ be a polynomial expression and $T$ be an orbit type with one degree of freedom such that for all $f \in \operatorname{Conf}_T$, $s$ is a factor of $f$ or shares no common factor with $f$. Then all orbits with locations in the zero set $S$ of $s$ are equivalent under $\equiv_c$.
\label{thm:irreducibleequivalence}
\end{lemma}
\begin{proof}
    Let $U,V \in T$ be two orbits whose locations $u,v$ are in the zero set of $s$. If $s$ is a factor of $f$, then $f(u)=f(v)=0$. Otherwise, the polynomials do not share a common root - i.e. $s(x) = 0 \implies f(x)\neq 0$. Letting $x = u$ or $v$, which we know are roots of $s$, we obtain $f(u) \neq 0$ and $f(v) \neq 0$. This implies $f(v)=0 \iff f(u)=0$ for all $f\in \operatorname{Conf}_T$, so $V\equiv_c U$ by definition.
 \end{proof}

Furthermore, if we know which entries of $\operatorname{Conf}_T$ have $s$ as a factor, it is possible to obtain the noble facetings of the equivalence class without needing to obtain a member of the class itself, as we know exactly which entries of the volume configuration evaluate to 0 for all orbits with locations in the zero set of $s$. Thus, if we can factor the polynomials of $\operatorname{Conf}_T$ into a set $S$ such that any $f \in \operatorname{Conf}_T$ can be formed as a product of members of $S$, and all polynomials in $S$ are irreducible, then all $s \in S$ must satisfy the condition given in Lemma 3.24. We give an example of a way to construct such a set below. Note that when using Definitions \ref{def:332matrices},\ref{def:432matrices},\ref{def:532matrices}, the coordinates of these orbit types are given by polynomials with coefficients in $\mathbb{Q}[\sqrt{2},\sqrt{5}]$.

\begin{definition}
    Let $T$ be an orbit type with at least one degree of freedom. Then we define the set of \textit{coprime polynomials} of $T$ to be $\operatorname{Copr}_T = \bigcup\limits_{f\in \operatorname{Conf}_T}\operatorname{fac}(f)$, where $\operatorname{fac}(f)$ is the set of irreducible nonconstant factors of $f$ under the field $\mathbb{Q}[\sqrt{2},\sqrt{5}]$.
\label{def:copr}
\end{definition}

\begin{theorem}
    For all $s \in \operatorname{Copr}_T$ and $f \in \operatorname{Conf}_T$ for some orbit type $T$, either $s$ is a factor of $f$ or shares no common factors with $f$.
\label{thm:coprfundamental}
\end{theorem}
\begin{proof}
    As $s$ is by definition an irreducible polynomial, it must be that $\gcd(s,f) \in \{s,1\}$. In the former case, $s$ must be a factor of $f$. In the latter case, $s$ and $f$ must share no common factors.
\end{proof}

\begin{corollary}
    If $T$ is an orbit type with one degree of freedom, all critical orbits have a location which is the root of exactly one polynomial $s \in \operatorname{Copr}_T$. Furthermore, all typical orbits have a location which is not the root of any member of $\operatorname{Copr}_T$.
\label{thm:main1dcopr}
\end{corollary}

This corollary is useful - as previously mentioned, from a member $s \in \operatorname{Copr}_T$ we may determine the set of facetings of the orbits that are located at roots of $s$, even when we do not know if $s$ actually has a positive real root or not. This is due to the fact that four points $\{v_1,v_2,v_3,v_4\}$ become coplanar at a root of $s$ exactly when the corresponding entry of $\operatorname{Conf}_T$ has $s$ as a factor. From that point, we may use a root-finding algorithm to determine if $s$ actually has any positive real roots, as each positive root will correspond to a nondegenerate noble polyhedron. This gives an effective procedure to find noble polyhedra within an orbit type with one degree of freedom, although the final implementation uses various other methods to improve on the time required for a computer program to complete the proof in these cases. We discuss ways that we reduce the number of computations in Section 5.

The problem is more difficult in the case of an orbit type $T$ with two degrees of freedom. The critical orbits of $T$ are precisely those which are roots of some element of $\operatorname{Copr}_T$; furthermore, two orbits will be critical-equivalent if their locations are roots of the same members of $\operatorname{Copr}_T$. As any two members $s,t \in \operatorname{Copr}_T$ do not share a common factor, their corresponding plane curves intersect at only a finite number of points. As such, the equivalence classes of $T$ under $\equiv_c$ are precisely the typical equivalence class, one equivalence class corresponding to each member of $\operatorname{Copr}_T$, and the equivalence classes of the orbits which lie on the intersection of the zero sets of two (or more) members of $\operatorname{Copr}_T$. In first two cases, we may again abstractly determine the noble facetings for these cases as we did for the one-dimensional case, and if it is the case that facetings exist, determine if the polynomial has any roots with $a,b > 0$. Through verification by computer search, it will turn out that no abstract facetings exist in these cases.

The final case is more difficult, as in general it is possible for multiple members of $\operatorname{Copr}_T$ to evaluate to zero at once in a given location (these are exactly the points where at least two zero sets intersect). For any two polynomials $s,t \in \operatorname{Copr}_T$, we may determine the location of the intersections of their zero sets through resultants; in particular, $\operatorname{Res}_b(s,t)$ is a polynomial whose roots give the $a$-coordinates of the intersections and $\operatorname{Res}_a(s,t)$ is a polynomial whose roots give the $b$-coordinates of the intersections.

\begin{definition}
    Let $T$ be an orbit type with 2 degrees of freedom. We define $\operatorname{Copr}^a_T$ to the the union \[\operatorname{Copr}^a_T = \bigcup\limits_{s,t \in \operatorname{Copr}_T} \operatorname{fac}(\operatorname{Res}_b(s,t)),\] where $\operatorname{fac}$ is defined as in Definition \ref{def:copr}. We define $\operatorname{Copr}^b_T$ similarly, with the $b$ in the definition replaced with an $a$.
\end{definition}

\begin{definition}
    Let $T$ be an orbit type with two degrees of freedom. We define an \textit{atlas of roots} of $a$ for $T$ to be a set of closed, disjoint intervals $\{I_0, I_1, I_2, ...\}$ such that each root $x$ of a member of some $s \in \operatorname{Copr}^a_T$ lies within some interval $I_n$, and each interval contains exactly one root from $\operatorname{Copr}^a_T$. We define an \textit{atlas of roots} of $b$ for $T$ similarly, with occurences of $a$ replaced with $b$.
\end{definition}

Having an atlas of roots is helpful in our enumeration, as it allows us to determine if a set of coprime polynomials $s_1, s_2, ..., s_k \in \operatorname{Copr}_T$ all intersect at the same point. If $\{I_0,I_1,...\}$ is an atlas of roots of $a$ and $\{J_0,J_1,...\}$ is an atlas of roots of $b$, then there is a shared intersection if $s_1, s_2, ..., s_k$ all have an intersection within $I_n \times J_m$ for some $n,m$. We discuss the exact methods used within our computer-assisted enumeration during Section 5.

\section{Prismatic symmetry}

\subsection{Existing examples}

We now move on to proving the task of enumerating the prismatic noble polyhedra, which is computationally feasible without computer assistance. There are two classes of previously known prismatic noble polyhedra: the \textit{disphenoids} and the \textit{stephanoids}. As mentioned in Section 1, every triangle can be used as the face of a unique noble tetrahedron (a disphenoid), and these are the only convex noble polyhedra that are not regular. If the disphenoid is composed of isosceles triangles, it is a \textit{tetragonal disphenoid}. If the disphenoid is composed of scalene triangles, it is known as a \textit{rhombic disphenoid}.

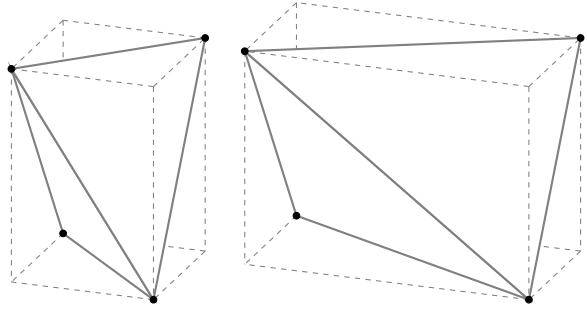
\begin{figure}[H]
    \centering

\tdplotsetmaincoords{70}{20}

\begin{tabular}{l r}
\begin{tikzpicture}[tdplot_main_coords]
    \ld{1,1,1.5}{-1,1,1.5}
    \ld{1,1,-1.5}{-1,1,-1.5}
    \ld{1,-1,-1.5}{-1,-1,-1.5}
    \ld{1,1,1.5}{1,-1,1.5}
    \ld{-1,1,1.5}{-1,-1,1.5}
    \ld{1,1,-1.5}{1,-1,-1.5}
    \ld{-1,1,-1.5}{-1,-1,-1.5}
    \ld{1,1,1.5}{1,1,-1.5}
    \ld{-1,1,1.5}{-1,1,-1.5}
    \ld{-1,-1,1.5}{-1,-1,-1.5}

    \draw[thick, gray, fill=white] (-1,-1,1.5) -- (-1,1,-1.5) -- (1,-1,-1.5) -- cycle;
    \draw[thick, gray, fill=white] (1,1,1.5) -- (-1,-1,1.5) -- (1,-1,-1.5) -- cycle;
    
    \ld{1,-1,1.5}{-1,-1,1.5}
    \ld{1,-1,1.5}{1,-1,-1.5}
    \ld{1,-1,1.5}{1,1,1.5}
    
    \ptb{1,1,1.5}
    \ptb{-1,-1,1.5}
    \ptb{-1,1,-1.5}
    \ptb{1,-1,-1.5}
\end{tikzpicture}\vspace{6pt} &

\begin{tikzpicture}[tdplot_main_coords]
    \ld{2,1,1.5}{-2,1,1.5}
    \ld{2,1,-1.5}{-2,1,-1.5}
    \ld{2,-1,-1.5}{-2,-1,-1.5}
    \ld{2,1,1.5}{2,-1,1.5}
    \ld{-2,1,1.5}{-2,-1,1.5}
    \ld{2,1,-1.5}{2,-1,-1.5}
    \ld{-2,1,-1.5}{-2,-1,-1.5}
    \ld{2,1,1.5}{2,1,-1.5}
    \ld{-2,1,1.5}{-2,1,-1.5}
    \ld{-2,-1,1.5}{-2,-1,-1.5}

    \draw[thick, gray, fill=white] (-2,-1,1.5) -- (-2,1,-1.5) -- (2,-1,-1.5) -- cycle;
    \draw[thick, gray, fill=white] (2,1,1.5) -- (-2,-1,1.5) -- (2,-1,-1.5) -- cycle;
    
    \ld{2,-1,1.5}{-2,-1,1.5}
    \ld{2,-1,1.5}{2,-1,-1.5}
    \ld{2,-1,1.5}{2,1,1.5}
    
    \ptb{2,1,1.5}
    \ptb{-2,-1,1.5}
    \ptb{-2,1,-1.5}
    \ptb{2,-1,-1.5}
\end{tikzpicture} \\

\end{tabular}
    
    \caption{An example of a tetragonal disphenoid (left) with isosceles triangle faces and a rhombic disphenoid (right) with scalene triangle faces.}
    \label{fig:disphenoidexamples}
\end{figure}

\textit{Stephanoids} or \textit{crown polyhedra} are nonconvex examples of prismatic noble polyhedra which consist of self-intersecting quadrilateral faces. The stephanoids are further subdivided into two categories: prismatic (whose vertices are in the Pn orbit types) and antiprismatic (whose vertices are in the An orbit types). We briefly repeat the definition of stephanoids described in \cite{hollowfaces} here.

The vertices of an orbit in a Pn orbit type may be partitioned uniquely into the vertices of two regular $n$-gons, from which we may label the points in this orbit counterclockwise as $a_1,a_2,...,a_n$ for the first $n$-gon and $b_1,b_2,...,b_n$ for the second $n$-gon (see Figure \ref{fig:stephanoidexamples}). Similarly, we may label the vertices of an orbit in an An orbit type counterclockwise as $a_1,a_2,...,a_{2n}$.

\begin{definition}
    A \textit{prismatic stephanoid} $\operatorname{PC}(n,p,q)$ with $n,p,q$ positive and $2p-n < 2q < p < n$ is a polyhedron in the P$n$ orbit type generated by the quadrilateral with vertices $a_1,b_{1+q},a_{1+p},b_{1+p-q}$ in order under $*22n$ symmetry.
\label{def:prismaticstephanoid}
\end{definition}

\begin{definition}
    An \textit{antiprismatic stephanoid} $\operatorname{AC}(n,p,q)$ with $n,p,q$ positive, $q$ odd, and $2p-n < q < p < n$ is a polyhedron in the A$n$ orbit type generated by the quadrilateral with vertices $a_1,a_{1+q},a_{1+2p},a_{1+2p-q}$ in order under $2{*}n$ symmetry.
\label{def:antiprismaticstephanoid}
\end{definition}

\begin{definition}
    A prismatic noble polyhedron $\mathcal{P}$ is a \textit{stephanoid} if it is of the form $\operatorname{PC}(n,p,q)$ or $\operatorname{AC}(n,p,q)$ for some positive integers $n,p,q$.
\label{def:stephanoid}
\end{definition}

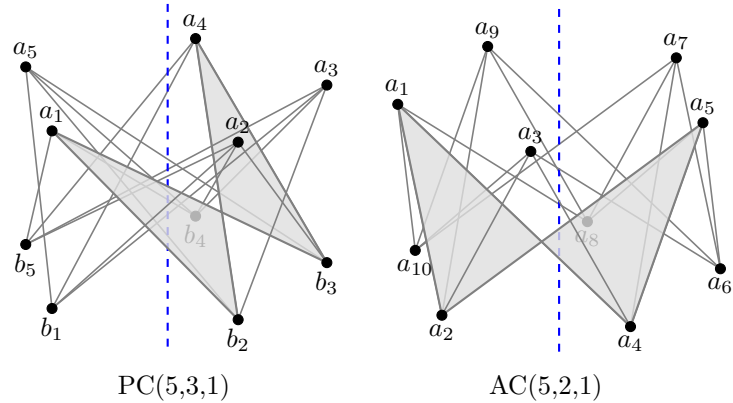
\begin{figure}[H]
    \centering

\tdplotsetmaincoords{70}{10}

\begin{tabular}{J J}

\begin{tikzpicture}[tdplot_main_coords, scale = 2.5]
    \ls{-0.809016, 0.2628655, -0.5}{-0.0, 0.8506508, 0.5}
    \ls{-0.0, 0.8506508, 0.5}{0.4999999, -0.688190, -0.5}
    \ls{-0.0, 0.8506508, 0.5}{0.8090169, 0.2628655, -0.5}
    \ls{-0.500000, -0.688190, -0.5}{-0.809016, 0.2628655, 0.5}
    \ls{-0.0, 0.8506508, -0.5}{0.25, 0.0812304, 0}
    \ls{0.154508, -0.212662, 0}{-0.500000, -0.688190, -0.5}
    \ls{-0.500000, -0.688190, 0.5}{-0.809016, 0.2628655, -0.5}
    \ls{-0.0, 0.8506508, 0.5}{-0.500000, -0.688190, -0.5}
    \ls{-0.0, 0.8506508, -0.5}{-0.500000, -0.688190, 0.5}
    \ls{-0.0, 0.8506508, -0.5}{-0.809016, 0.2628655, 0.5}
    \ls{-0.809016, 0.2628655, 0.5}{0.4999999, -0.688190, -0.5}
    \ls{0.8090169, 0.2628655, -0.5}{-0.809016, 0.2628655, 0.5}
    \ls{-0.809016, 0.2628655, -0.5}{0.00257386, -0.326789, 0.12}
    \ls{0, 0.262866, 0}{-0.809016, 0.2628655, -0.5}
    \ls{-0.0, 0.8506508, -0.5}{0.8090169, 0.2628655, 0.5}

    \draw[thick, blue, dashed] (0,0,1) -- (0,0,-0.97);

    \pt{-0.0, 0.8506508, -0.5}
    \node[below=0pt of {(0.0, 0.8506508, -0.5)}, outer sep=0pt] {$b_4$};
    
    \draw[thick, gray, fill=lightgray!50!white, fill opacity = 0.8] (-0.500000, -0.688190, 0.5) -- (0.4999999, -0.688190, -0.5) -- (0.0, 0.8506508, 0.5) -- (0.8090169, 0.2628655, -0.5) -- cycle;
    
    \ls{0.25, 0.0812304, 0}{0.4999999, -0.688190, 0.5}
    \ls{0.8090169, 0.2628655, 0.5}{0.154508, -0.212662, 0}
    \ls{0.8090169, 0.2628655, -0.5}{0.4999999, -0.688190, 0.5}
    \ls{0.00257386, -0.326789, 0.12}{0.4999999, -0.688190, 0.5}
    \ls{-0.500000, -0.688190, -0.5}{0.4999999, -0.688190, 0.5}
    \ls{0.8090169, 0.2628655, 0.5}{0, 0.262866, 0}
    \ls{0.8090169, 0.2628655, 0.5}{0.4999999, -0.688190, -0.5}
    
    \node[above=0pt of {(-0.500000, -0.688190, 0.5)}, outer sep=0pt] {$a_1$};
    \node[below=0pt of {(-0.500000, -0.688190, -0.5)}, outer sep=0pt] {$b_1$};
    \node[above=0pt of {(0.4999999, -0.688190, 0.5)}, outer sep=0pt] {$a_2$};
    \node[below=0pt of {(0.4999999, -0.688190, -0.5)}, outer sep=0pt] {$b_2$};
    \node[above=0pt of {(0.8090169, 0.2628655, 0.5)}, outer sep=0pt] {$a_3$};
    \node[below=0pt of {(0.8090169, 0.2628655, -0.5)}, outer sep=0pt] {$b_3$};
    \node[above=0pt of {(0.0, 0.8506508, 0.5)}, outer sep=0pt] {$a_4$};
    \node[above=0pt of {(-0.809016, 0.2628655, 0.5)}, outer sep=0pt] {$a_5$};
    \node[below=0pt of {(-0.809016, 0.2628655, -0.5)}, outer sep=0pt] {$b_5$};
    
    \pt{-0.500000, -0.688190, 0.5}
    \pt{0.8090169, 0.2628655, 0.5}
    \pt{-0.809016, 0.2628655, -0.5}
    \pt{-0.0, 0.8506508, 0.5}
    \pt{-0.500000, -0.688190, -0.5}
    \pt{0.8090169, 0.2628655, -0.5}
    \pt{-0.809016, 0.2628655, 0.5}
    \pt{0.4999999, -0.688190, 0.5}
    \pt{0.4999999, -0.688190, -0.5}
\end{tikzpicture} &

\begin{tikzpicture}[tdplot_main_coords, scale = 2.5]
    \ls{-0.506731, 0.6974565, 0.5067318}{0.8199093, 0.2664047, -0.506731}
    \ls{0.5067318, 0.6974565, 0.5067318}{0.5067318, -0.697456, -0.506731}
    \ls{-0.506731, 0.6974565, 0.5067318}{-0.506731, -0.697456, -0.506731}
    \ls{-0.819909, -0.266404, 0.5067318}{-0.0, 0.8621037, -0.506731}
    \ls{0.8199093, -0.266404, 0.5067318}{0.8199093, 0.2664047, -0.506731}
    \ls{-0.819909, 0.2664047, -0.506731}{0.5067318, 0.6974565, 0.5067318}
    \ls{-0.0, 0.8621037, -0.506731}{0.8199093, -0.266404, 0.5067318}
    \ls{0.5067318, 0.6974565, 0.5067318}{0.8199093, 0.2664047, -0.506731}
    \ls{-0.819909, -0.266404, 0.5067318}{-0.819909, 0.2664047, -0.506731}
    \ls{-0.819909, -0.266404, 0.5067318}{0.5067318, -0.697456, -0.506731}
    \ls{-0.0, 0.8621037, -0.506731}{-0.506731, 0.6974565, 0.5067318}
    \ls{-0.819909, 0.2664047, -0.506731}{-0.506731, 0.6974565, 0.5067318}
    \ls{-0.819909, -0.266404, 0.5067318}{-0.506731, -0.697456, -0.506731}
    \ls{-0.0, 0.8621037, -0.506731}{0.5067318, 0.6974565, 0.5067318}
    \ls{0.5067318, -0.697456, -0.506731}{0.8199093, -0.266404, 0.5067318}

    \draw[thick, blue, dashed] (0,0,1) -- (0,0,-0.97);
    
    \ls{-0.819909, 0.2664047, -0.506731}{-0.0, -0.862103, 0.5067318}
    \ls{-0.0, -0.862103, 0.5067318}{0.8199093, 0.2664047, -0.506731}
    
    \pt{-0.0, 0.8621037, -0.506731}
    \node[below=0pt of {(-0.0, 0.8621037, -0.506731)}, outer sep=0pt] {$a_8$};
    
    \draw[thick, gray, fill=lightgray!50!white, fill opacity = 0.8] (-0.819909, -0.266404, 0.5067318) -- (-0.506731, -0.697456, -0.506731) -- (0.8199093, -0.266404, 0.5067318) -- (0.5067318, -0.697456, -0.506731) -- cycle;
    
    \ls{-0.0, -0.862103, 0.5067318}{0.5067318, -0.697456, -0.506731}
    \ls{-0.0, -0.862103, 0.5067318}{-0.506731, -0.697456, -0.506731}

    \node[below=0pt of {(-0.819909, 0.2664047, -0.506731)}, outer sep=0pt] {$a_{10}$};
    \node[above=0pt of {(-0.506731, 0.6974565, 0.5067318)}, outer sep=0pt] {$a_9$};
    \node[above=0pt of {(0.5067318, 0.6974565, 0.5067318)}, outer sep=0pt] {$a_7$};
    \node[below=0pt of {(0.8199093, 0.2664047, -0.506731)}, outer sep=0pt] {$a_6$};
    \node[above=0pt of {(0.8199093, -0.266404, 0.5067318)}, outer sep=0pt] {$a_5$};
    \node[below=0pt of {(0.5067318, -0.697456, -0.506731)}, outer sep=0pt] {$a_4$};
    \node[above=0pt of {(0.0, -0.862103, 0.5067318)}, outer sep=0pt] {$a_3$};
    \node[below=0pt of {(-0.506731, -0.697456, -0.506731)}, outer sep=0pt] {$a_2$};
    \node[above=0pt of {(-0.819909, -0.266404, 0.5067318)}, outer sep=0pt] {$a_1$};

    \pt{-0.819909, -0.266404, 0.5067318}
    \pt{-0.819909, 0.2664047, -0.506731}
    \pt{0.5067318, 0.6974565, 0.5067318}
    \pt{-0.506731, 0.6974565, 0.5067318}
    \pt{-0.0, -0.862103, 0.5067318}
    \pt{-0.506731, -0.697456, -0.506731}
    \pt{0.5067318, -0.697456, -0.506731}
    \pt{0.8199093, -0.266404, 0.5067318}
    \pt{0.8199093, 0.2664047, -0.506731}
\end{tikzpicture}\vspace{1pt} \\

PC(5,3,1) & AC(5,2,1) \\

\end{tabular}
    
    \caption{Examples of prismatic and antiprismatic stephanoids along with the labelings of their vertices. One face of each polyhedron is highlighted along with all other edges, and the main axis of symmetry of each polyhedron is indicated in blue.}
    \label{fig:stephanoidexamples}
\end{figure}

These definitions are equivalent to those in \cite{hollowfaces}, where it is shown that these all generate valid prismatic noble polyhedra satisfying the requirements given in Section 2. If $n,p,q$ share a common factor, then this process generates a compound of smaller stephanoids.

\subsection{Completeness of enumeration}


\begin{definition}
    The \textit{main axis of symmetry} of a prismatic symmetry group $*22n, 2{*}n, 22n,n*,$ or $n\times$ is an axis of order-$n$ rotational symmetry. This axis is unique if $n > 2$.
\end{definition}

Using this definition, all prismatic orbit types have a unique main axis of symmetry other than A2, sA2, and tP2 (in these cases, there are three possible main axes of symmetry).

An important property of this axis is that scaling an orbit $V$ along the main axis of symmetry (even when it is not unique) simply multiplies the outputs of the volume configuration $\operatorname{Conf}_V$ by a constant factor, and the transformed version of $V$ is still within the same orbit type. As a result, any prismatic noble polyhedron $\mathcal{P}$ is part of an infinite continuum of prismatic noble polyhedra which differ only by the amount of scaling along the main axis of symmetry. Thus without loss of generality we may assume that all prismatic noble polyhedra are of unit height.

If we let $V$ be a member of a prismatic orbit type, then the set $P(V)$ of planes (as in Definition \ref{def:planes}) may be partitioned as follows:
\begin{itemize}
    \item 2 \textit{base planes} which each contain half the points in the orbit. These are the only planes which are perpendicular to the main axis of symmetry, and they only exist when the main axis of symmetry is unique.
    \item Planes containing two points from one of the base planes, and one point from the other.
    \item A plane containing two points from both base planes.
\end{itemize}
Notably, only the base planes may contain more than 4 points (if a plane intersected a base plane at at least three points, it would have the same affine hull as the base plane and so it must be the same plane). Because there are only two base planes, they trivially cannot contain the face of a noble polyhedron and therefore a prismatic noble polyhedron has faces with either 3 or 4 sides.

\begin{theorem}
    The dual $\mathcal{Q}$ of a prismatic noble polyhedron $\mathcal{P}$ is a prismatic noble polyhedron.
\end{theorem}
\begin{proof}
    Because duality preserves symmetry and the first and third conditions of Definition \ref{def:nondegenerate} are trivially satisfied, we only need to prove that the realization is faithful and no two coplanar faces of $\mathcal{Q}$ share an edge. If the realization was not faithful, then two vertices of $\mathcal{Q}$ must realize to the same point in $\mathbb{R}^3$, which occurs exactly when two faces of $\mathcal{P}$ are coplanar. In order for the faces not to share an edge, the plane containing these two faces must contain 
 at least 5 vertices, which is impossible. Similarly, if $\mathcal{Q}$ had coplanar faces, then the realization of the dual $\mathcal{P}$ would have coinciding vertices and therefore $\mathcal{P}$ would not be faithful.
\end{proof}

It thus follows that a prismatic noble polyhedron will either have 3 or 4 faces to a vertex as its dual must have 3 or 4 sides to a face. This fact is useful in combination with the following:

\begin{lemma}
    Given a noble polyhedron $\mathcal{P}$ composed of $v$ vertices such that each face has $k$ edges and the order of the stabilizer of a face is $\varphi$, the number of faces $N$ around a vertex is given by

    \[N = \frac{|\Gamma(\mathcal{P})|k}{\varphi v} \ge 3.\]
\label{thm:prismaticformula}
\end{lemma}
\begin{proof}
    Let $f$ be the number of faces of $\mathcal{P}$. Then $f = \frac{|\Gamma(\mathcal{P})|}{\varphi}$ by the definition of stabilizers as noble polyhedra are face-transitive. If we imagine that the faces of $\mathcal{P}$ were disconnected and isolated, we would obtain a figure containing $fk$ vertices. When joined together into a polyhedron, the vertices merge together in groups of $N$ and so the number of vertices in our polyhedron is $v = \frac{fk}{N}$. From this we find that $N = \frac{fk}{v}$, which using the definition of $f$ is equivalent to the first part of our claim. As we cannot have 2 faces or less at a vertex in a polyhedron, $N \ge 3$.
\end{proof}
\begin{corollary}
    If $\mathcal{P}$ is a noble polyhedron, then $N = \frac{k}{\varphi}$ except when $\mathcal{P}$ is in a Pn orbit type under $*22n$ symmetry or in an An orbit type under $2{*}n$ symmetry. In these cases, $N=2\frac{k}{\varphi}$.
\label{thm:prismaticN}
\end{corollary}
\begin{proof}
    This may be confirmed manually by simply checking all possibilities for orbit types and the point groups that act transitively on them. Tables \ref{table:prismatictypes} and \ref{table:prismaticorders} (seen below) give the necessary information to do this.
\end{proof}
\vspace{-0.7cm}
\begin{table}[H]
    \caption{Orders of the prismatic point groups which act transitively on an orbit type.}
    \vspace{0.05cm}
  \centering
  \begin{tabular}{ | l | l | }
    \hline
    Orbifold symbol & Order \\ \hline \hline
    $n*$      & $2n$ \\ \hline
    $n\times$ & $2n$ \\ \hline
    $22n$     & $2n$ \\ \hline
    $2{*}n$     & $4n$ \\ \hline
    $*22n$    & $4n$ \\ \hline
  \end{tabular}
\label{table:prismaticorders}
\end{table}

This corollary is very powerful in combination with the previous fact that $N,k$ are either 3 or 4. The case $k = 3$ is simple, and by duality this also enumerates the prismatic noble polyhedra with $N = 3$. We show this below.

\begin{theorem}
    The only prismatic noble polyhedra with triangular faces are the family of disphenoids.\label{thm:prismatictrianglefaced}
\end{theorem}
\begin{proof}
    In the case where $N = \frac{3}{\varphi} \ge 3$, the only possible value for $\varphi$ is 1 and so $N = 3$. In the case where $N = \frac{6}{\varphi} \ge 3$, again the only possibilities for $\varphi$ are 1 or 2. However if $\varphi = 1$ then $N = 6$, which is impossible as the dual of this polyhedron would be prismatic and have hexagonal faces.  
    
    Therefore in either case we have $k = N = 3$, and because the tetrahedron is the only abstract polyhedron with 3 triangles at a vertex a prismatic noble polyhedron with triangular faces must have the abstract structure of a tetrahedron. If $\varphi = 1$, faces will have no symmetry which results in a rhombic disphenoid. If $\varphi = 2$ faces will have mirror symmetry which results in a tetragonal disphenoid.
\end{proof}

The only remaining case is when $N = k = 4$, and we want to show that a prismatic noble polyhedron of this form must be a stephanoid. Without loss of generality, we may assume that these polyhedra have a unique main axis of symmetry. The only cases where the choice of the main axis of symmetry would not be unique are the orbits of A2, sA2, and tP2. Of these, only tP2 has a plane containing 4 points, and these planes either intersect the origin (which is impossible) or lie on the exterior of the polyhedron, which would result in a convex noble polyhedron - however the convex hull of tP2 is never noble under $*222$ symmetry; it will have the symmetries of a rectangular prism, which has three orbits of faces.

\begin{theorem}
    Let $\mathcal{P}$ be a prismatic noble polyhedron with quadrilateral faces. Then no edge of $\mathcal{P}$ may be contained within a base plane.
    \label{thm:noedgesinbaseplane}
\end{theorem}
\begin{proof}
    Assume a face $F$ of $\mathcal{P}$ contains an edge $e_1$ within a base plane. Therefore $F$ must also contain a second edge $e_2$ within the opposite base plane. $e_1$ must also be adjacent to a unique second face $F'$ of $\mathcal{P}$, so because $\mathcal{P}$ is noble and therefore face-transitive, there must be a nonidentity symmetry $\rho \in \Gamma(\mathcal{P})$ such that $\rho(F)=F'$. Then we have either $\rho(e_1) = e_1$ or $\rho(e_1)=e_2$. In the latter case, $\rho(e_1)$ would equal some transformation of $e_2$ under a symmetry of $\Gamma(\mathcal{P})$, causing $e_1$ and $e_2$ to have the same length. This would imply that either all faces intersect the origin, or all face planes are perpendicular to the base planes; in either case, this would cause $F'$ and $F$ to coincide.

    Therefore $\rho(e_1)=e_1$. It follows that $\rho$ must either swap or preserve the vertices of $e_1$ in its transformation. In the latter case, symmetries in $*22n$ which do this only exist when $e_1$ intersects with the main axis of symmetry. In the former case, $e_1$ either intersects the main axis of symmetry or $\rho$ is a reflection, however when $\rho$ is a reflection it must preserve the face plane of $F$ - this causes $F'$ and $F$ to be adjacent faces in the same face plane, violating Definition \ref{def:nondegenerate}. Thus in either case $e_1$ must intersect the main axis of symmetry, and by extension the same logic must apply for $e_2$. However, in this case $F$ must intersect the origin, which is impossible.
\end{proof}

\begin{lemma}
    Let $\mathcal{P}$ be a noble polyhedron in the P$n$ orbit type with $*22n$ symmetry. If it contains a face that visits the vertices $a_1,b_i,a_j,b_k$ in order for some $i,j,k$, and that face has a mirror symmetry containing the main axis of symmetry, it is a stephanoid.
\label{thm:prismaticcondition}
\end{lemma}
\begin{proof}
    Let $\rho$ be the mirror symmetry of the face. A mirror symmetry containing the main axis of symmetry cannot swap the base planes of $\mathcal{P}$, so we must have $\rho(a_1) = a_j, \rho(b_i)=b_k$. This uniquely determines the plane of symmetry as the plane of points that are an equal distance to $a_1$ and $a_j$ (or equivalently $b_i$ and $b_k$). It may be verified that $\rho(b_p)=b_{1+j-p}$ for all $p$, and therefore $k = 1+j-i$. This is the generating face of $\operatorname{PC}(n,j-1,i-1)$ (if it exists).

    The two cycles $a_1,b_i,a_j,b_k$ and $a_1,b_k,a_j,b_i$ give the same face, so without loss of generality we may assume that $i < k = 1+j-i$, or equivalently $2i-2 < j-1$.

    Consider the map $f$ such that $f(a_p)=b_{i-p+1 \pmod n}$ and $f(b_p)=a_{i-p+1 \pmod n}$. This transformation extends to a symmetry of $*22n$, and it maps the vertices of the face to $a_1, b_i, a_{n+2i-j}, b_{n+i-j+1}$. Let $j'=n+2i-j$. Without loss of generality, we may assume that $2j-n \le 2i$, as if instead $2j-n \ge 2i$, it would follow that $2j'-n=4i-(2j-n) \le 4i-(2i) = 2i$, so application of $f$ gives another face of $\mathcal{P}$ which satisfies the condition. In order to prove the necessary relation $2j-n-2 < 2i-2 < j-1 < n$ (see Definition \ref{def:prismaticstephanoid}), it only remains to show that $2j-n \neq 2i$. If this was the case, then geometrically the face will intersect the origin. However, this is impossible; if we let $e$ be an edge of the chosen face which does not intersect the center of the polyhedron, then the face plane of the one other adjacent face containing $e$ will contain both $e$ and the origin; this results in two faces in the same plane sharing an edge, which is not allowed.

    Thus $\mathcal{P}$ must be the polyhedron $\operatorname{PC}(n,j-1,i-1)$.
\end{proof}

\begin{lemma}
    Let $\mathcal{P}$ be a noble polyhedron in the A$n$ orbit type with $2{*}n$ symmetry. If it contains a face that visits the vertices $a_1,a_{2i},a_{2j+1},a_{2k}$ in order for some integral $i,j,k$, and that face has a mirror symmetry containing the axis of symmetry, it is a stephanoid.
\label{thm:antiprismaticcondition}
\end{lemma}
\begin{proof}
    Let $\rho$ be the mirror symmetry of the face. $\rho$ cannot swap the base planes of the vertices of $\mathcal{P}$, so if $a_q = \rho(a_p)$ both $p$ and $q$ have the same parity. Thus $\rho(a_1)=a_{2j+1}$ and $\rho(a_{2i})=a_{2k}$, and this uniquely implies $\rho(a_p)=a_{2+2j-p}$ in general; furthermore, we obtain the identity $k=1+j-i$ from this. This gives the generating face of $\operatorname{AC}(n,j,2i-1)$ (if it exists). $2i-1$ will always be odd, so it only remains to verify that without loss of generality we may assume $2j-n < 2i-1 < j < n$ as stated in Definition \ref{def:antiprismaticstephanoid}.

    $j < n$ is trivially true as otherwise the chosen face would visit $a_1$ twice. As in the previous lemma, the two cycles $a_1,a_{2i},a_{2j+1},a_{2k}$ and $a_1,a_{2k},a_{2j+1},a_{2i}$ give the same polygon, so without loss of generality we may assume $2i < 2k = 2+2j-2i$, or equivalently $2i - 1 < j$.

    Consider the map $f$ such that $f(a_p) = a_{p-1-2j+2i}$. This transformation is a symmetry of $2{*}n$, and it maps the vertices of the face to $a_1,a_{2i},a_{4i-2j-1},a_{2i-2j}$. Let $j'=2i-j-1$. If $2j-n\ge 2i-1$,  then $2j'-n = 4i-2-(2j-n)-2n \le 2i-1-2n < 2i-1$. Thus without loss of generality we may assume $2j-n < 2i-1$, as otherwise we may apply $f$ to obtain another face of $\mathcal{P}$ which does satisfy the condition.

    Thus, $\mathcal{P}$ must be the polyhedron $\operatorname{AC}(n,j,2i-1)$.
\end{proof}

\begin{theorem}
    Let $\mathcal{P}$ be a noble polyhedron, $F$ be a face of $\mathcal{P}$, and $\rho$ be a symmetry of the polygon $F$. Then $\rho$ is also a symmetry of $\mathcal{P}$.
\label{thm:facesymmetryextends}
\end{theorem}
\begin{proof}
    As $\mathcal{P}$ is noble and therefore face-transitive, all faces of $\mathcal{P}$ are of the form $g(F)$ for some $g \in \Gamma(\mathcal{P})$. However, $\rho(g(F))$ will be a face of $\mathcal{P}$ if and only if the conjugation $g(\rho(g(g^{-1}(F))))=g(\rho(F))=g(F)$ is also a face of $\mathcal{P}$, which it clearly is. Therefore $\rho$ maps faces of $\mathcal{P}$ to faces of $\mathcal{P}$, and clearly because of this it will map edges to edges and vertices to vertices as well.
\end{proof}

Note that the previous lemma works for all noble polyhedra, however it is only useful in the prismatic case so we include it here instead of in previous sections.

\begin{theorem}
    Let $\mathcal{P}$ be a prismatic noble polyhedron. If $N = k = 4$, then $\varphi = 2$.
\label{thm:quadsymmetry}
\end{theorem}
\begin{proof}
    As previously mentioned, by Corollary \ref{thm:prismaticN} we have either $N = \frac{k}{\varphi}$ or $N = 2\frac{k}{\varphi}$ for all prismatic noble polyhedra. Taking $N = k = 4$, it follows that either $\varphi = 2$ or $\varphi = 1$, with the latter indicating that the faces of $\mathcal{P}$ have no symmetry under $\Gamma(\mathcal{P})$; we show that this case is impossible. All edges must also contain one vertex from each base plane as a consequence of Theorem \ref{thm:noedgesinbaseplane}.
    
    Let $F$ be a face of $\mathcal{P}$ connected in order by edges $e_1, e_2, e_3, e_4$. We argue that no two of these edges may have equal length. Without loss of generality there are only two possible cases if it was true that two edges of $F$ were the same length:

    \begin{itemize}
        \item $e_1 = e_2$. $e_1$ and $e_2$ together contain three vertices $v_1, v_2, v_3$, two of which (say, $v_1$ and $v_2$) are in the same base plane. The final vertex $v_4$ of $F$ must lie on the intersection between the plane $p = \operatorname{aff}(\{v_1,v_2,v_3\})$, the circumsphere of $\mathcal{P}$, and the other base plane containing $v_3$. However, the only intersection point between these three surfaces is $v_3$, so there is no possible location where $v_4$ could exist.
        \item $e_1 = e_3$. In this case, $e_1$ and $e_3$ share no vertices so the locations of the two edges completely determine $F$. No matter how $e_3$ is placed relative to $e_1$, $F$ will either have $180^\circ$ rotational symmetry about its center or a reflection symmetry due to both edges having one vertex in each base plane. Both cases contradict the assumption that $\varphi = 1$ by Theorem \ref{thm:facesymmetryextends}.
    \end{itemize}
    
    As $\mathcal{P}$ is noble, for each edge there must exist isometries $\rho_1, \rho_2, \rho_3, \rho_4 \in \Gamma(\mathcal{P})$ where $\rho_i$ maps $F$ to the unique face adjacent to $F$ that contains $e_i$. Because all edges are unequal in length, it follows that $\rho_i(e_i) = e_i$ for each $i$. There are three possible isometries for each symmetry $\rho_i$: reflection about the plane containing $e_i$ and the main axis of symmetry, reflection about the plane through the midpoint of $e_i$ and perpendicular to the main axis of symmetry, or reflection about the line containing the midpoint of $e_i$ and perpendicular to the main axis of symmetry. All reflections of a prismatic symmetry group are either perpendicular to or contain the main axis of symmetry, so in either case if $\rho_i$ is a reflection $e_i$ must be parallel to the axis of symmetry (it cannot be perpendicular due to Lemma 4.8). We then have three possibilities for the nature of the $\rho_i$:

    \begin{itemize}
        \item If two of the $\rho_i$ are reflections, then $F$ has two edges parallel to the main axis of symmetry, and the resulting bowtie-shaped face has $180^\circ$ symmetry; this is a contradiction again due to Theorem \ref{thm:facesymmetryextends}.
        \item If one of the $\rho_i$ is a reflection, then by vertex-transitivity each vertex must be adjacent to each of the four types of edges. This implies that $\sigma = \rho_1\rho_2\rho_3\rho_4$ is the identity, however $\sigma$ is not an orientation-preserving transformation and so this cannot be the case. 
        \item If all of the $\rho_i$ are rotations, then we must generate the group $\Gamma(\mathcal{P})$ solely from rotations. The rotation axes of the $\rho_i$ all lie in a common plane parallel to the base planes, implying that the midpoints of the $e_i$ are collinear. Because $F$ is concyclic, this implies that $F$ has a mirror symmetry $\rho'$, again leading to a contradiction.
    \end{itemize}\vspace{-0.75cm}
\end{proof}
\begin{corollary}
    Any prismatic noble polyhedron with quadrilateral faces either has vertices in the Pn orbit type or the An orbit type. In these cases, their symmetry groups must be $*22n$ and $2{*}n$ respectively.
\label{thm:1dtypereduction}
\end{corollary}
\begin{proof}
    By Corollary \ref{thm:prismaticN}, we must have $N=\frac{k}{\varphi}$ or $2\frac{k}{\varphi}$. If $N=4, k=4,$ and $\varphi=2$, the latter must be the case. It may be manually checked that the specified situations are the only cases where $N=2\frac{k}{\varphi}$.
\end{proof}

\begin{theorem}
    Any prismatic noble polyhedron $\mathcal{P}$ with quadrilateral faces is a stephanoid.
\label{thm:quad1d}
\end{theorem}
\begin{proof} 
    We consider the Pn and An cases separately. 

    In the Pn case, Lemma \ref{thm:noedgesinbaseplane} implies that the faces of $\mathcal{P}$ self-intersect and its edges alternate between the base planes. Consider one of the four faces of $\mathcal{P}$ containing the vertex $a_1$. This face must visit the vertices $a_1,b_i,a_j,b_k$ in order for some $i,j,k$ as the edges of $\mathcal{P}$ alternate between the base planes. Then for Lemma \ref{thm:prismaticcondition} to apply, we must show that the face has a mirror symmetry containing the main axis of symmetry and that the polyhedron is noble under $*22n$ symmetry - the latter is given by Corollary \ref{thm:1dtypereduction}.

    Theorem \ref{thm:quadsymmetry} shows that any face $F$ of $\mathcal{P}$ must have exactly one nonidentity symmetry $\rho$, and therefore that symmetry must be an involution. As $\rho \in \Gamma(\mathcal{P}) = *22n$ for some $n$, we have 3 possibilities for $\rho$ as $F$ is not within a base plane: reflection parallel to a base plane, a $180^\circ$ rotation around the center of $F$, or a reflection containing the main axis of symmetry. $F$ is concyclic due to $\mathcal{P}$ being vertex-transitive, so in either of the first two cases the only possibility for the face of $F$ is a crossed rectangle as the vertices of $\mathcal{P}$ lie within one of the two base planes. However, $\varphi > 2$ in these cases as crossed quadrilaterals have more than one nonidentity symmetry. Lemma \ref{thm:prismaticcondition} shows that $\mathcal{P}$ is a stephanoid.

    In the An case, Lemma \ref{thm:noedgesinbaseplane} again implies that its faces self-intersect and its edges alternate between the base planes. Consider one of the four faces of $\mathcal{P}$ that contains the vertex with label $a_1$. This face $F$ must visit the vertices $a_1,a_i,a_j,a_k$ in order for some $i,j,k$, where $i$ and $k$ are even and $j$ is odd to have edges alternate between base planes. In order for Lemma \ref{thm:antiprismaticcondition} to apply, we must show that $F$ contains a mirror symmetry containing the axis of symmetry and $\mathcal{P}$ is noble under $2{*}n$ symmetry - the latter is given by Corollary \ref{thm:1dtypereduction}. Here, the same logic as used previously applies in order to show $F$ has a mirror symmetry containing the main axis of symmetry, as $2{*}n \subseteq *22n$. Therefore Lemma \ref{thm:antiprismaticcondition} shows that $\mathcal{P}$ is a stephanoid.
\end{proof}

\begin{corollary}
    Any prismatic noble polyhedron $\mathcal{P}$ is either a disphenoid or a stephanoid.
\label{thm:prismaticenumeration}
\end{corollary}
\begin{proof}
    If $\mathcal{P}$ has triangular faces, Theorem \ref{thm:prismatictrianglefaced} shows it is a disphenoid. If $\mathcal{P}$ has quadrilateral faces, Theorem \ref{thm:quad1d} shows it is a stephanoid. $\mathcal{P}$ cannot have faces with a higher number of sides as these faces would need to be contained within the base planes of $\mathcal{P}$, which will always result in a degenerate noble polyhedron.
\end{proof}

\section{Enumeration}

\subsection{Implementation}

As previously mentioned, we will now use the tools set out in Section 3.3 to design and implement an algorithm that checks, for a given orbit type $T$, the set of all facetings for each equivalence class of $T$ under $\equiv_c$. The difficulty of the enumeration varies depending on the number of degrees of freedom of $T$.

\begin{itemize}
    \item If $T$ has zero degrees of freedom, the problem is trivial - Table \ref{table:nonprismatictypes} gives possible exact coordinates for the single orbit $V$ in the orbit type, and so we may simply check each plane $p \in P(V)$ for noble facetings using the adjacency graph strategy described in Section 3.1.
    \item If $T$ has one degree of freedom, we may symbolically determine $\operatorname{Conf}_T$ and $\operatorname{Copr}_T$, which gives exactly one equivalence class under $\equiv_c$ for each $s \in \operatorname{Copr}_T$, where the set of positive real roots of $s$ is an equivalence class. It is possible to abstractly determine the set of planes for each equivalence class, so we check for facetings in each of these and also the typical equivalence class. Once we have a set of abstract "candidate" noble polyhedra we may determine if they may be realized by approximating the positive real roots of $s$ using a root-finding algorithm; in our implementation the Jenkins-Traub algorithm is used. Each root $x$ corresponds to a valid realization of the abstract noble polyhedron onto the vertices of $T(x)$.
    \item If $T$ has two degrees of freedom, we may again symbolically determine $\operatorname{Conf}_T$ and $\operatorname{Copr}_T$. Once again, each member of $\operatorname{Copr}_T$ corresponds to an equivalence class under $\equiv_c$ and we may abstractly check if any member of that equivalence class has facetings (and the same applies for the typical equivalence class). However, there also exist additional equivalence classes at the intersections of one or more members of $\operatorname{Copr}_T$, and we solve the problem of determining which subsets of $\operatorname{Copr}_T$ intersect simultaneously using numerical methods. If we also symbolically calculate $\operatorname{Copr}^a_T$ and $\operatorname{Copr}^b_T$, we may numerically approximate the roots for each polynomial in $\operatorname{Copr}^a_T$ or $\operatorname{Copr}^b_T$ with a root-finding algorithm (again, we use the Jenkins-Traub algorithm in our implementation) until the intervals given by the upper and lower bounds no longer overlap with each other - this gives atlases of roots of both $a$ and $b$ for $T$. We may use these atlases of roots to exactly determine the remaining equivalence classes under $\equiv_c$, as if $I$ is an interval in an atlas of roots for $a$ and $J$ is an interval in an atlas of roots for $b$, then $I \times J$ may contain at most one intersection point between two or more members of $\operatorname{Copr}_T$. Determining if any intersection exists at the given point was done through the Wolfram Language's \texttt{Reduce} operation, which was sufficient to determine the existence of an intersection in all cases that were required. We may then find the facetings of each equivalence class, which completes the enumeration.
\end{itemize}

Our implementation uses a combination of Python's \texttt{sympy} and \texttt{numpy} libraries, which were used for faceting calculations and basic symbolic manipulation, as well as the Wolfram Language, which processed much of the more intensive symbolic and numeric operations. The full commentated codebase is available on a GitHub repository made by the present author \cite{noblepolyhedrarepository}, which additionally has a library of 3d model files in the \texttt{.off} file format for all noble polyhedra.

\subsection{Results}

The results of the computer-assisted enumeration are summarized in Tables \ref{table:firstappendixtable}-\ref{table:lastappendixtable}. Each noble polyhedron orbit is assigned a symbol $T$-$x$ where $T$ is an orbit type and $x$ is an integer; this indicates that $T$-$x$ is the $x$-th orbit of $T$ containing facetings, where orbits within $T$ are sorted lexicographically based on their locations (when $T$ has two degrees of freedom, the value of the $a$ parameter is treated as more significant). Each noble polyhedron is assigned a symbol $T$-$x.y$ which indicates it is the $y$-th faceting of the orbit $T$-$x$; here, the facetings are ordered by inradius from largest to smallest. In the event that multiple facetings have the same inradius, the polyhedra are sorted by the order in which they were found by the computer enumeration. The only exception to this naming scheme is when $T$ has zero degrees of freedom, where we simply label the orbit as $T$ rather than $T$-$1$ (there is only one possible orbit, so this loses no information).

In total, there are 146 nonprismatic noble polyhedra up to similarity which satisfy Definition \ref{def:nondegenerate}: one each in the orbit types T,O,C,tO,tC, and rC, 4 in the I orbit type, 6 in the ID orbit type, 7 in the D orbit type, 17 in the tI orbit type, 6 in the tD orbit type, 19 in the rD orbit type, 7 in the sC orbit type, 3 in the gC orbit type, 33 in the sD orbit type, and 38 in the gD orbit type. Notably, there are no noble polyhedra in the CO, tT, rT, rP, sT, gT, or gP orbit types. Along with Corollary \ref{thm:prismaticenumeration}, this proves Theorem \ref{thm:main}.

All but four noble polyhedra have polyhedral duals, with the exceptions being gD-19.1, gD-28.1, D-4, and D-5. The polyhedra in these four cases contain faces which are coplanar, but do not meet at an edge. As a result, their duals contain distinct vertices which coincide. These duals may either be interpreted as containing coinciding vertices, or as an object whose vertex figure is a compound polygon. We will call a shape with this property a \textit{fissary} polyhedron - note that these are only polyhedra when interpreted in the first sense, and even then these polyhedra are degenerate. The fissary noble polyhedra are listed below in Table \ref{table:fissarypolyhedra}, where the top row for each polyhedron interprets it as a degenerate polyhedron with coinciding vertices, and the bottom row interprets it as an object with a compound vertex figure. We assign each of these symbols, but these are somewhat arbitrary.
\vspace{-0.25cm}
\begin{table}[h!]
  \caption{The four fissary noble polyhedra.}
  \centering
  \begin{tabular}{ | l | l | l | l | l | l | }
    \hline
    Symbol & Vertices & Edges & Faces & Schläfli type & Dual \\ \hline \hline
    \multirow{2}{3em}{tI-F} & 120 & \multirow{2}{2em}{300} & \multirow{2}{2em}{120} & $\{5,5\}$ & \multirow{2}{3.5em}{gD-19.1} \\ \cline{2-2} \cline{5-5}
                            & 60 &  &  & $\{5,10/2\}$ & \\ \hline
    \multirow{2}{3em}{rD-F} & 120 & \multirow{2}{2em}{300} & \multirow{2}{2em}{120} & $\{5,5\}$ & \multirow{2}{3.5em}{gD-28.1} \\ \cline{2-2} \cline{5-5}
                & 60 &  &  & $\{5,10/2\}$ & \\ \hline
    \multirow{2}{3em}{D-F1} & 60 & \multirow{2}{2em}{120} & \multirow{2}{2em}{20} & $\{12,4\}$ & \multirow{2}{3.5em}{D-4} \\ \cline{2-2} \cline{5-5}
                & 20 &  &  & N/A & \\ \hline
    \multirow{2}{3em}{D-F2} & 60 & \multirow{2}{2em}{90} & \multirow{2}{2em}{20} & $\{9,3\}$ & \multirow{2}{3.5em}{D-5} \\ \cline{2-2} \cline{5-5}
                & 20 &  &  & N/A & \\ \hline
  \end{tabular}
\label{table:fissarypolyhedra}
\end{table}

D-F1 and D-F2 have faces that visit multiple vertices in identical locations, even though abstractly these vertices may be considered distinct. When interpreted in the latter sense, the faces of D-F1 and D-F2 are not abstract polygons, so the Schläfli type is undefined. This is not the case for tI-F or rD-F, whose faces are pentagons in either interpretation.

\subsection{Generalizations}
The fact that the definition of 'polyhedron' has been inconsistent from author to author naturally leads to the question of which additional noble polyhedra may exist when the requirements on what counts as a polyhedron are relaxed. We briefly discuss possible generalizations of the enumeration problem, the effectiveness of the above methods for these enumerations, and the feasibility of a full enumeration.

\begin{itemize}
    \item \textit{Allowing infinite polyhedra.} Allowing infinitely large polyhedra would almost surely result in many infinite sequences of polyhedra of arbitrary complexity due to the fact that polygons can have an unbounded and even infinite number of sides. Furthermore, due to the infinite number of vertices that these polyhedra would have, the methods used in this paper would not be applicable for computer search. Perhaps an enumeration would be more feasible if we restrict the number of sides that the face of a polyhedron may have (e.g. finding infinite noble polyhedra with specifically triangular faces).
    \item \textit{Allowing nonfaithful realizations.} Allowing nonfaithful polyhedra is uninteresting without further restrictions as any abstract noble polyhedron $(P,\le)$ may have all its vertices realized onto the origin, giving a valid nonfaithful "noble polyhedron". However, the existence of fissary polyhedra (as listed in Table \ref{table:fissarypolyhedra}) could hint at the existence of an interesting relaxation of this condition, as fissary polyhedra can be interpreted as nonfaithful noble polyhedra.
    \item \textit{Allowing skew faces.} Allowing nonplanar (skew) faces turns the enumeration into a purely combinatorial problem as determining if sets of points are coplanar is no longer a concern, and so the above methods will not be very helpful. Definitions allowing these, like the skeletal polyhedra defined in \cite{oldandnew}, have become popular when analyzing regular polytopes and led to interesting results (see e.g. \cite{schulteskeletal} and Chapter 5 of \cite{ARP}). It seems natural to extend the idea to noble polyhedra as well as they are a generalization of regular polyhedra. Unfortunately, because so many more generating faces are possible (in theory the faces of nonprismatic polyhedra could have up to 120 sides!) it's likely that there are many thousands, if not millions, of skew noble polyhedra up to the canonical mapping given in Definition \ref{def:canonicalmapping}. However, it would be very interesting if that was incorrect.
    \item \textit{Allowing coplanar faces sharing an edge.} Removing the final condition of Definition \ref{def:nondegenerate} may still lead to an interesting problem, as many of the theorems in this paper would still apply. The faceting algorithm would require some fundamental changes, and all of the possible new polyhedra under this definition would become a noble polyhedron or noble polyhedron compound under the current definition when coplanar faces are combined together at their shared edges.
    \item \textit{Reducing restrictions on the abstract structure.} Many past examples of noble polyhedra such as the V-faced polyhedra of \cite{hollowfaces} abstractly do not have a polyhedral structure due to the fact that their faces revisit vertices or some related issue. Some formalisms exist which allow these possibilities, such as the definition of polyhedra described in \cite{hollowfaces}, so separate enumerations under these conditions would be possible (although the total number of noble polyhedra would likely be much larger, similarly to when skew faces are allowed). A separate question that may also lead to an interesting problem is the enumeration of noble \textit{compound polyhedra} - compounds of polyhedra that are transitive on their faces and vertices. For both of these generalizations the methods used in this paper would likely be applicable, and the main required change would be to adjust the faceting algorithm.
    \item \textit{Investigating other ranks of polytopes.} One may more generally define a \textit{noble $n$-polytope} as an $n$-polytope that acts transitively on its vertices and facets. Even in the case of 4-polytopes, the computational difficulty of calculating facetings becomes immense, especially for 4-polytopes with a large number of vertices such as the 120-cell $\{5,3,3\}$, due to the fact that the number of possible planes increases exponentially as the dimension becomes large. A possible restriction that would result in an interesting question is the problem of enumerating the \textit{convex} noble $n$-polytopes for some $n > 3$. It's possible that a higher-dimensional generalization of the above methods could be used for this, but finding orbit types in dimensions greater than 3 would also be much more difficult as they are not all subgroups of Coxeter groups.
\end{itemize}

\section{Acknowledgments}

I would like to sincerely thank Dr. Jan Reimann for his guidance and feedback during the development of this proof. Additional thanks to the members of the Polytope Discord, who both inspired me to tackle this problem and kept me motivated to continue.

\pagebreak

\printbibliography

@book{ARP,
    author    = "Peter McMullen and Egon Schulte",
    title     = "Abstract Regular Polytopes",
    year      = "2002",
    publisher = "Cambridge University Press",
    series = "Encyclopedia of Mathematics and Its Applications",
    address   = "Cambridge, United Kingdom"
}

@article{hessinitial,
    author = "Edmund Hess",
    title = "{zwei Erweiterungen des Begriffs der regelm{\"a}ssigen K{\"o}rper}. ({German})
    [Two extensions to the concept of regular bodies]",
    journal = "der Gesellshaft zur Bedf{\"o}rdung der gesammten Naturwissenschaften",
    pages = "3-26",
    year = "1875"
}

@book{hess1,
    author = "Edmund Hess",
    title = "{Ueber die zugleich gleicheckigen und gleichfl{\"a}chigen Polyeder}. ({German})
    [About the simultaneously vertex-transitive and face-transitive polyhedra]",
    series = "zwei Erweiterungen des Begriffs der regelm{\"a}ssigen K{\"o}rper",
    year      = "1876",
    publisher = "Kay"
}

@article{hess2,
    author = "Edmund Hess",
    title = "{Ueber einige merkw{\"u}rdige nicht convexe Polyeder}. ({German})
    [About some strange nonconvex polyhedra]",
    journal = "der Gesellshaft zur Bedf{\"o}rdung der gesammten Naturwissenschaften",
    pages = "3-17",
    year = "1877"
}

@book{bruckner1,
    author = "Max Br{\"u}ckner",
    title = "{Vielecke und Vielflache: Theorie und Geschichte}. ({German})
    [Polygons and Polyhedra: Theory and History]",
    year      = "1900",
    publisher = "Teubner"
}

@article{bruckner2,
    author = "Max Br{\"u}ckner",
    title = "{{\"U}ber die diskontinuierlichen und nicht-konvexen gleicheckig-gleichfl{\"a}chigen Polyeder}. ({German})
    [On the nonconvex and vertex-transitive, face-transitive polyhedra]",
    journal = "Verh. des dritten Internat. Math.-Kongresses",
    year = "1905"
}

@article{bruckner3,
    author = "Max Br{\"u}ckner",
    title = "{{\"U}ber die gleicheckig-gleichfl{\"a}chigen diskontinuierlichen und nichtkonvexen Polyeder}. ({German})
    [On the equal-angle, equal-surface, discontinuous, and non-convex polyhedra]",
    year = "1906",
    journal = "Nova Acta Leop."
}

@article{bruckner4,
    author = "Max Br{\"u}ckner",
    title = "{Zur Geschichte der Theorie der gleicheckig
gleichfl{\"a}chigen Polyeder}. ({German})
    [On the history of the theory of vertex-transitive face-transitive polyhedra]",
    journal = "Unterrichtsblatter Math. Naturwiss.",
    year = "1907"
}

@article{oldandnew,
  title="Regular polyhedra—old and new",
  author="Gr{\"u}nbaum, Branko",
  journal="Aequationes mathematicae",
  volume="16",
  number="1",
  pages="1--20",
  year="1977",
  publisher="Springer"
}

@incollection{hollowfaces,
    author = "Branko Gr{\"u}nbaum",
    editor = "T. Bisztriczky and P. McMullen and R. Schneider and A. I. Weiss",
    title = "Polyhedra with hollow faces",
    booktitle = "POLYTOPES: Abstract, Convex and Computational",
    publisher = "Kluwer Academic Publishers",
    year = "1993",
    volume = "440",
    pages = "43-70",
}

@incollection{yourpolymypoly,
    author = "Branko Gr{\"u}nbaum",
    editor = "B. Aronov and S. Basu and J. Pach and M. Sharir",
    title = "Are your polyhedra the same as my polyhedra?",
    booktitle = "Discrete and Computational Geometry: The Goodman-Pollack Festschrift",
    publisher = "Springer",
    year = "2003",
    pages = "461-488"
}

@article{tetvolume,
    author = "H. B. Newson",
    title = "On the volume of a polyhedron.",
    journal = "Annals of Mathematics",
    year = "1899",
    pages = "108-110",
    volume = "1",
    number = "1/4",
    publisher={JSTOR}
}

@article{orbifoldnotation,
    author = "John H. Conway and Daniel H. Huson",
    title = "The Orbifold Notation for Two-Dimensional Groups",
    journal = "Structural Chemistry",
    year = "2002",
    volume = "13",
    number = "3-4",
    pages = "247-257"
}

@inproceedings{stella4Dexploration,
    author = "Ulrich Mikloweit",
    editor = "Carolyn Yackel and Eve Torrence and Kristóf Fenyvesi and Robert Bosch and Craig Kaplan",
    title = "Exploring Noble Polyhedra With the Program Stella4D",
    booktitle = "Bridges 2020 Conference Proceedings",
    year = "2020",
    volume = "25",
    pages = "257-264"
}

@article{schulteskeletal,
  title="Skeletal geometric complexes and their symmetries",
  author="Schulte, Egon and Weiss, Asia Ivi{\'c}",
  journal="arXiv preprint arXiv:1610.02619",
  year="2016"
}

@online{noblepolyhedrarepository,
  author = "Connor Hill",
  title = {GitHub repository: noble-tools-revised},
  url = {https://github.com/Plasmath/noble-tools-revised},
  year = "2025",
  note = {Accessed: 2025-09-11}
}

\pagebreak

\section*{Appendix A: List of Noble Polyhedra}

\begin{longtable}[c]{| c | c | c | c | c | c | c |} 

 \caption{Noble polyhedra in orbit types with 0 degrees of freedom.}\label{table:firstappendixtable}\\

 \hline
 Symbol & Schläfli type & Vertices & Edges & Faces & Symmetry & Dual \\
 \hline
 \endfirsthead

 \hline
 Symbol & Schläfli type & Vertices & Edges & Faces & Symmetry & Dual \\
 \hline
 \endhead

 \hline
 \endfoot

 \hline
 \endlastfoot

 T-1 & $\{3,3\}$ & 4   & 6   & 4   & $*332$ & T-1  \\
 O-1 & $\{3,4\}$ & 6   & 12  & 8   & $*432$ & C-1  \\
 C-1 & $\{4,3\}$ & 8   & 12  & 6   & $*432$ & O-1  \\
 I-1 & $\{3,5\}$ & 12  & 30  & 20  & $*532$ & D-1  \\
 I-2 & $\{5,5\}$ & 12  & 30  & 12  & $*532$ & I-3  \\
 I-3 & $\{5,5\}$ & 12  & 30  & 12  & $*532$ & I-2  \\
 I-4 & $\{3,5\}$ & 12  & 30  & 20  & $*532$ & D-6  \\
ID-1 & $\{4,8\}$ & 30  & 120 & 60  & $*532$ & tI-5.1 \\
ID-2 & $\{3,12\}$& 30  & 180 & 120 & $*532$ & gD-27.1 \\
ID-3 & $\{4,8\}$ & 30  & 120 & 60  & $*532$ & tI-3.1 \\
ID-4 & $\{3,12\}$& 30  & 180 & 120 & $*532$ & gD-23.1 \\
ID-5 & $\{4,8\}$ & 30  & 120 & 60  & $*532$ & rD-5.2 \\
ID-6 & $\{3,12\}$& 30  & 180 & 120 & $*532$ & gD-22.1 \\
 D-1 & $\{5,3\}$ & 20  & 30  & 12  & $*532$ & I-1  \\
 D-2 & $\{9,3\}$ & 20  & 90  & 60  & $*532$ & rD-5.1 \\
 D-3 & $\{6,6\}$ & 20  & 60  & 20  & $*532$ & D-3  \\
 D-4 & $\{4,12\}$& 20  & 120 & 60  & $*532$ & D-F1  \\
 D-5 & $\{3,9\}$ & 20  & 90  & 60  & $532$  & D-F2  \\
 D-6 & $\{5,3\}$ & 20  & 30  & 12  & $*532$ & I-4  \\
 D-7 & $\{5,3\}$ & 20  & 90  & 60  & $*532$ & tI-5.7 \\
 
 \end{longtable}

\begin{longtable}[c]{| c | c | c | c | c | c | c |}
 \caption{Noble polyhedra in orbit types with 1 degree of freedom.\label{long}}\\

 \hline
 Symbol & Schläfli type & Vertices & Edges & Faces & Symmetry & Dual \\
 \hline
 \endfirsthead

 \hline
 Symbol & Schläfli type & Vertices & Edges & Faces & Symmetry & Dual \\
 \hline
 \endhead

 \hline
 \endfoot

 \hline
 \endlastfoot

 tO-1.1 & $\{4,8\}$ & 24  & 96  & 48  & $*432$ & gC-3.1 \\
 tC-1.1 & $\{5,5\}$ & 24  & 60  & 24  & $*432$ & rC-1.1 \\
 rC-1.1 & $\{5,5\}$ & 24  & 60  & 24  & $*432$ & tC-1.1 \\
 tI-1.1 & $\{6,6\}$ & 60  & 180 & 60  & $*532$ & rD-3.1 \\
 tI-1.2 & $\{5,5\}$ & 60  & 150 & 60  & $532$  & rD-3.2 \\
 tI-2.1 & $\{6,6\}$ & 60  & 180 & 60  & $*532$ & tD-3.2 \\
 tI-2.2 & $\{6,6\}$ & 60  & 180 & 60  & $*532$ & rD-4.2 \\
 tI-3.1 & $\{8,4\}$ & 60  & 120 & 30  & $*532$ & ID-3 \\
\end{longtable}

 \begin{longtable}[c]{| c | c | c | c | c | c | c |}
 \caption{Noble polyhedra in the gC and sC orbit types.\label{long}}\\

 \hline
 Symbol & Schläfli type & Vertices & Edges & Faces & Symmetry & Dual \\
 \hline
 \endfirsthead

 \hline
 Symbol & Schläfli type & Vertices & Edges & Faces & Symmetry & Dual \\
 \hline
 \endhead

 \hline
 \endfoot

 \hline
 \endlastfoot

 sC-1.1 & $\{5,5\}$ & 24 & 60  & 24 & $432$  & sC-6.1 \\
 sC-2.1 & $\{5,5\}$ & 24 & 60  & 24 & $432$  & sC-4.1 \\
 sC-3.1 & $\{5,5\}$ & 24 & 60  & 24 & $432$  & sC-3.1 \\
 sC-4.1 & $\{5,5\}$ & 24 & 60  & 24 & $432$  & sC-2.1 \\
 sC-5.1 & $\{5,5\}$ & 24 & 60  & 24 & $432$  & sC-5.1 \\
 sC-6.1 & $\{5,5\}$ & 24 & 60  & 24 & $432$  & sC-1.1 \\
 sC-6.2 & $\{5,5\}$ & 24 & 60  & 24 & $432$  & sC-6.2 \\
 gC-1.1 & $\{5,5\}$ & 48 & 120 & 48 & $*432$ & gC-2.1 \\
 gC-2.1 & $\{5,5\}$ & 48 & 120 & 48 & $*432$ & gC-1.1 \\
 gC-3.1 & $\{8,4\}$ & 48 & 96  & 24 & $*432$ & tO-1.1 \\
 
 \end{longtable}

 \begin{longtable}[c]{| c | c | c | c | c | c | c |}
 \caption{Noble polyhedra in the sD orbit type.\label{long}}\\

 \hline
 Symbol & Schläfli type & Vertices & Edges & Faces & Symmetry & Dual \\
 \hline
 \endfirsthead

 \hline
 Symbol & Schläfli type & Vertices & Edges & Faces & Symmetry & Dual \\
 \hline
 \endhead

 \hline
 \endfoot

 \hline
 \endlastfoot

 sD-1.1  & $\{5,5\}$ & 60  & 150 & 60  & $532$ & sD-11.1 \\
 sD-2.1  & $\{5,5\}$ & 60  & 150 & 60  & $532$ & sD-17.1 \\
 sD-3.1  & $\{5,5\}$ & 60  & 150 & 60  & $532$ & sD-23.1 \\
 sD-4.1  & $\{5,5\}$ & 60  & 150 & 60  & $532$ & sD-28.1 \\
 sD-5.1  & $\{5,5\}$ & 60  & 150 & 60  & $532$ & sD-5.2 \\
 sD-5.2  & $\{5,5\}$ & 60  & 150 & 60  & $532$ & sD-5.1 \\
 sD-5.3  & $\{6,6\}$ & 60  & 180 & 60  & $532$ & sD-5.3 \\
 sD-6.1  & $\{5,5\}$ & 60  & 150 & 60  & $532$ & sD-13.1 \\
\end{longtable}

 \begin{longtable}[c]{| c | c | c | c | c | c | c |}
 \caption{Noble polyhedra in the gD orbit type.\label{long}}\\

 \hline
 Symbol & Schläfli type & Vertices & Edges & Faces & Symmetry & Dual \\
 \hline
 \endfirsthead

 \hline
 Symbol & Schläfli type & Vertices & Edges & Faces & Symmetry & Dual \\
 \hline
 \endhead

 \hline
 \endfoot

 \hline
 \endlastfoot

 gD-1.1  & $\{5,5\}$ & 120 & 300 & 120 & $*532$ & gD-25.1 \\
 gD-2.1  & $\{5,5\}$ & 120 & 300 & 120 & $*532$ & gD-38.1 \\
 gD-3.1  & $\{5,5\}$ & 120 & 300 & 120 & $*532$ & gD-17.1 \\
 gD-4.1  & $\{5,5\}$ & 120 & 300 & 120 & $*532$ & gD-36.1 \\
 gD-5.1  & $\{5,5\}$ & 120 & 300 & 120 & $*532$ & gD-35.1 \\
 gD-6.1  & $\{5,5\}$ & 120 & 300 & 120 & $*532$ & gD-29.1 \\
 gD-7.1  & $\{5,5\}$ & 120 & 300 & 120 & $*532$ & gD-20.1 \\
 gD-8.1  & $\{5,5\}$ & 120 & 300 & 120 & $*532$ & gD-18.1 \\
\end{longtable}

\section*{Appendix B: Noble Polyhedron Orbits}

\bgroup
\def\arraystretch{1.5}
 \begin{longtable}[c]{| c | c | c |}
 \caption{Minimal polynomials for the locations of orbits of noble polyhedra in orbit types with 1 degree of freedom.\label{long}}\\

 \hline
 Orbit & Approx. location & Minimal polynomial \\
 \hline
 \endfirsthead

 \hline
 Orbit & Approx. location & Minimal polynomial \\
 \hline
 \endhead

 \hline
 \endfoot

 \hline
 \endlastfoot

 tO-1 & 2.14789903570479 & $a^3 - a^2 - 2a - 1$ \\
 tC-1 & 1.00000000000000 & $a - 1$ \\
 rC-1 & 1.00000000000000 & $a - 1$ \\
 tI-1 & 0.25599806014774 & $a^4 - 2a^3 - 7a^2 - 2a + 1$ \\
 tI-2 & 0.27201964951406 & $a^4 + 4a^3 + 5a^2 + 2a - 1$ \\
 tI-3 & 0.61803398874989 & $a^2 + a - 1$ \\
 tI-4 & 0.86676039917386 & $a^4 + 2a^3 - a - 1$ \\
 tI-5 & 1.61803398874989 & $a^2 - a - 1$ \\
 tI-6 & 2.61803398874989 & $a^2 - 3a + 1$ \\
 tI-7 & 3.02048177471563 & $a^4 - a^3 - 5a^2 - 3a - 1$ \\
 tD-1 & 0.53568738679187 & $a^4 - a^3 + 2a - 1$ \\
 tD-2 & 1.00000000000000 & $a - 1$ \\
 tD-3 & 1.27201964951407 & $a^4 - a^2 - 1$ \\
 tD-4 & 2.61803398874989 & $a^2 - 3a + 1$ \\
 rD-1 & 0.30411353597566 & $a^4 - 8a^3 + 4a - 1$ \\
 rD-2 & 0.61803398874989 & $a^2 + a - 1$ \\
 rD-3 & 0.70710678118654 & $2a^2 - 1$ \\
 rD-4 & 0.78615137775742 & $a^4 + a^2 - 1$ \\
 rD-5 & 1.00000000000000 & $a - 1$ \\
 rD-6 & 1.15372137554177 & $a^4 + a^3 - 2a - 1$ \\
 rD-7 & 1.86676039917386 & $a^4 - 2a^3 + a - 1$ \\
 rD-8 & 2.41421356237309 & $a^2 - 2a - 1$ \\ \hline
 tI-F & 0.43168341659057 & $a^4 + 2a^3 + 2a^2 + a - 1$ \\
 rD-F & 1.43168341659058 & $a^4 - 2a^3 + 2a^2 - a - 1$
 
 \end{longtable}
\egroup

\bgroup
\def\arraystretch{1.5}
 \begin{longtable}[c]{| c | c | c |}
 \caption{Minimal polynomials for the locations of orbits of noble polyhedra in sC and gC orbit types.\label{long}}\\

 \hline
 Orbit & Approx. location & Minimal polynomial \\
 \hline
 \endfirsthead

 \hline
 Orbit & Approx. location & Minimal polynomial \\
 \hline
 \endhead

 \hline
 \endfoot

 \hline
 \endlastfoot

 \multirow{2}{2.2em}{sC-1}
 & $0.28863340344794$ & $a^6 + 24a^2 - 2$ \\
 & $0.36110308052865$ & $b^3 - b^2 + 3b - 1$ \\ \hline
 \multirow{2}{2.2em}{sC-2}
 & $0.34324491690878$ & $a^8 - 10a^6 + 14a^4 - 10a^2 + 1$ \\
 & $1.81735402102397$ & $b^4 - 3b^2 - 1$ \\ \hline
 \multirow{2}{2.2em}{sC-3}
 & $0.41421356237310$ & $a^2 + 2a - 1$ \\
 & $0.41421356237310$ & $b^2 + 2b - 1$ \\ \hline
 \multirow{2}{2.2em}{sC-4}
 & $0.41994644680253$ & $9a^8 - 8a^6 + 3a^4 - 6a^2 + 1$ \\
 & $1.22346006048543$ & $3b^4 + 2b^3 - 3b^2 - 4b - 1$ \\ \hline
 \multirow{2}{2.2em}{sC-5}
 & $0.91612594942735$ & $a^6 + 2a^4 - 2$ \\
 & $0.54368901269208$ & $b^3 + b^2 + b - 1$ \\ \hline
 \multirow{2}{2.2em}{sC-6}
 & $2.21352383971935$ & $a^6 - 4a^4 - 4a^2 - 2$ \\
 & $1.76929235423863$ & $b^3 - 2b - 2$ \\ \hline
 \multirow{2}{2.2em}{gC-1}
 & $0.32423426733000$ & $8a^{10} + 32a^8 + 48a^6 + 28a^4 + 6a^2 - 1$ \\
 & $1.18085147819018$ & $2b^5 + 4b^4 + 2b^3 - 4b^2 - 6b - 3$ \\ \hline
 \multirow{2}{2.2em}{gC-2}
 & $0.34645087475949$ & $a^{10} - 4a^6 + 6a^4 + 16a^2 - 2$ \\
 & $2.33028392884366$ & $b^5 - 2b^4 - 4b^2 + 6b - 2$ \\ \hline
 \multirow{2}{2.2em}{gC-3}
 & $0.48247861707892$ & $8a^6 + 8a^4 + 2a^2 - 1$ \\
 & $1.46557123187677$ & $b^3 - b^2 - 1$ \\
 
 \end{longtable}
\egroup

\bgroup
\def\arraystretch{1.5}
 \begin{longtable}[c]{| c | c | c |}
 \caption{Minimal polynomials for the locations of orbits of noble polyhedra in the sD orbit type.\label{long}}\\

 \hline
 Orbit & Approx. location & Minimal polynomial \\
 \hline
 \endfirsthead

 \hline
 Orbit & Approx. location & Minimal polynomial \\
 \hline
 \endhead

 \hline
 \endfoot

 \hline
 \endlastfoot

 \multirow{2}{2.2em}{sD-1}
 & $0.11016974327929$ & $a^8 + 7a^7 + 12a^6 + 13a^5 - 39a^4 - 89a^3 - 84a^2 - 35a + 5$ \\
 & $1.27703280443828$ & $b^8 + 5b^7 + 16b^6 + 41b^5 + 59b^4 + 3b^3 - 116b^2 - 135b - 45$ \\ \hline

 \multirow{2}{2.2em}{sD-2}
 & $0.15888207437274$ & $a^8 - 5a^7 - 7a^6 + 22a^5 - a^4 - 32a^3 - 22a^2 - 2a + 1$ \\
 & $0.09085618795516$ & $b^8 - 9b^7 - 61b^6 - 81b^5 - 14b^4 + 47b^3 + 29b^2 + 8b - 1$ \\ \hline

 \multirow{2}{2.2em}{sD-3}
 & $0.18996264849311$ & $a^8 - 3a^7 - 13a^6 - 5a^5 + 11a^4 - 7a^2 - 4a + 1$ \\
 & $1.27703280443828$ & $b^8 + 16b^7 + 30b^6 + 24b^5 - 44b^4 + 10b^3 + 21b^2 + 2b - 1$ \\ \hline

 \multirow{2}{2.2em}{sD-4}
 & $0.35261409753226$ & $a^4 + 4a^3 - 2a^2 + 3a - 1$ \\
 & $1.40010649156259$ & $b^4 - 5b^3 + 3b^2 - 5b + 11$ \\
\hline
 
 \end{longtable}
\egroup

\begin{landscape}

\bgroup
\def\arraystretch{1.5}
 \begin{longtable}[c]{| c | c | c |}
 \caption{Minimal polynomials for the locations of orbits of noble polyhedra in the gD orbit type.
\label{table:lastappendixtable}}\\

 \hline
 Orbit & Approx. location & Minimal polynomial \\
 \hline
 \endfirsthead

 \hline
 Orbit & Approx. location & Minimal polynomial \\
 \hline
 \endhead

 \hline
 \endfoot

 \hline
 \endlastfoot

 \multirow{2}{2.2em}{gD-1}
 & $0.01108372009481$ & $5a^{10} + 55a^9 + 151a^8 + 24a^7 - 474a^6 - 651a^5 - 136a^4 + 176a^3 - 22a^2 - 90a + 1$ \\
 & $0.11212795108700$ & $5b^{10} + 40b^9 + 81b^8 - 26b^7 - 133b^6 - 52b^5 + 41b^4 + 50b^3 + 20b^2 + 6b - 1$ \\ \hline

 \multirow{2}{2.2em}{gD-2}
 & $0.02380983028619$ & $11a^{12} - 13a^{11} - 310a^{10} - 894a^9 - 947a^8 - 47a^7 + 578a^6 + 117a^5 - 380a^4 - 202a^3 + 47a^2 + 41a - 1$ \\
 & $0.24392803043993$ & $11b^{12} + 59b^{11} + 211b^{10} + 543b^9 + 1053b^8 + 1497b^7 + 1518b^6 + 1008b^5 + 339b^4 - 40b^3 - 75b^2 - 10b + 5$ \\ \hline

 \multirow{2}{2.2em}{gD-3}
 & $0.04741988306984$ & $a^4 - 3a^3 + 2a^2 + 21a - 1$ \\
 & $0.61803398874989$ & $b^2 + b - 1$ \\ \hline

 \multirow{2}{2.2em}{gD-4}
 & $0.05118010640774$ & $9a^{12} - 81a^{10} - 234a^9 - 312a^8 - 180a^7 + 79a^6 + 210a^5 + 144a^4 + 52a^3 + 27a^2 + 18a - 1$ \\
 & $2.63878671374419$ & $9b^{12} + 72b^{11} + 180b^{10} + 24b^9 - 936b^8 - 2844b^7 - 4792b^6 - 5208b^5 - 3744b^4 - 1616b^3 - 288b^2 + 48b + 16$ \\
\hline
 
 \end{longtable}
\egroup

\end{landscape}

\end{document}